\documentclass{mod-rmta}

\usepackage{xcolor}
\definecolor{webgreen}{rgb}{0,.5,0}
\definecolor{webbrown}{rgb}{.6,0,0}
\definecolor{RoyalBlue}{cmyk}{1, 0.50, 0, 0}
\usepackage[colorlinks=true, breaklinks=true, urlcolor=webbrown, linkcolor=RoyalBlue, citecolor=webgreen,backref=page]{hyperref}
\usepackage{amssymb,amsmath,amsfonts,amscd,mathrsfs,pifont,upgreek}

\newcommand{\qq}[1]{\qquad \mbox{#1} \qquad}

\newcommand{\BB}[1]{\ensuremath{\mathbb{#1}}}

\newcommand{\N}{\ensuremath{\BB{N}}}

\newcommand{\R}{\ensuremath{\BB{R}}}
\newcommand{\Z}{\ensuremath{\BB{Z}}}
\newcommand{\C}{\ensuremath{\BB{C}}}
\newcommand{\T}{\ensuremath{\BB{T}}}
\newcommand{\D}{\ensuremath{\BB{D}}}

\newcommand{\upd}{\ensuremath{\mathrm{d}}}

\newcommand{\la}{\ensuremath{\langle}}
\newcommand{\ra}{\ensuremath{\rangle}}

\newcommand{\transpose}{\ensuremath{\mathsf{T}}}

\newcommand{\G}[1]{\Gamma\left( #1 \right)}

\DeclareMathOperator{\sgn}{sgn}

\DeclareMathOperator{\Pf}{Pf}

\DeclareMathOperator{\im}{im}
\DeclareMathOperator{\re}{re}

\newcommand{\rec}{\ensuremath{\mathrm{rec}}}

\begin{document}

\markboth{C.D. Sinclair and M.L. Yattselev}{The reciprocal Mahler ensembles of random polynomials}

\title{The reciprocal Mahler ensembles of random polynomials}

\author{Christopher D.~Sinclair}
\address{Department of Mathematics, University of Oregon, Eugene OR 97403 \\ 
\email{\href{mailto:csinclai@uoregon.edu}{csinclai@uoregon.edu}} }

\author{Maxim L. Yattselev}

\address{Department of Mathematical Sciences, Indiana University-Purdue University, 402 North Blackford Street, Indianapolis IN 46202\\
\href{mailto:maxyatts@iupui.edu}{\tt maxyatts@iupui.edu} }

\maketitle

\begin{abstract}
We consider the roots of uniformly chosen complex and real reciprocal
polynomials of degree $N$  whose Mahler measure is bounded by a constant.
After a change of variables this reduces to a generalization of
Ginibre's complex and real ensembles of random matrices where the
weight function (on the eigenvalues of the matrices) is replaced by
the exponentiated equilibrium potential of the interval $[-2,2]$ on
the real axis in the complex plane.  In the complex (real) case the random
roots form a determinantal (Pfaffian) point process, and in both cases
the empirical measure on roots converges weakly to the arcsine
distribution supported on $[-2,2]$.  Outside this region the kernels
converge without scaling, implying among other things that there is
a positive expected number of outliers away from $[-2,2]$.  These
kernels, as well as the scaling limits for the kernels in the bulk
$(-2,2)$ and at the endpoints $\{-2,2\}$ are presented.  These kernels
appear to be new, and we compare their behavior with related kernels which arise
from the (non-reciprocal) Mahler measure ensemble of random
polynomials as well as the classical Sine and Bessel kernels.
\end{abstract}


\ccode{Mathematics Subject Classification 2000: 
15B52, %random matrices (algebraic aspects)
60B20, %random matrices (probabilistic aspects)
60G55,  %Point processes
82B23, %statistical mechanics: Exactly solvable models; Bethe ansatz
15A15, %Determinants, permanents, other special matrix functions
%11C08  %Number Theory Polynomials
11C20  %Number Theory Matrices, determinants
%11K36  %Number Theory Well-distributed sequences and other variations
%11R42  %Zeta functions and $L$-functions of number fields
%11R04  %Algebraic numbers; rings of algebraic integers
%11R06  %PV-numbers and generalizations; other special algebraic
}

\section{Introduction}	

We study two ensembles of random polynomials/matrices related to
Ginibre's real and complex ensembles of random matrices, but with
weight functions which are not derived from `classical' polynomial
potentials, but rather from the equilibrium logarithmic potential of
the interval $[-2,2]$ on the real axis in the complex plane.  This
complements our study of similar ensembles formed from the equilibrium
logarithmic potential of the unit disk \cite{MR2903124, SYreal}.

We introduce the complex ensemble first since it is simpler to
define.  Consider a joint density function of $N$ complex random
variables identified with $\mathbf z \in \C^N$ given by $P_N : \C^N
\rightarrow [0, \infty)$, where
\begin{equation}
\label{eq:0}
P_N(\mathbf z) = \frac{1}{Z_N}\bigg\{\prod_{n=1}^N \phi(z_n)\bigg\}
|\Delta(\mathbf z)|^2 \qq{and} \Delta(\mathbf z) = \prod_{m < n}
(z_n - z_m).
\end{equation}
Here $\phi : \C \rightarrow [0,\infty)$ is the {\em weight function},
and $Z_N$ is a constant necessary to make $P_N$ a probability density.
When $\phi(z) = e^{-|z|^2}$, $P_N$ gives the joint density of
eigenvalues of a matrix chosen randomly from Ginibre's ensemble of $N \times
N$ complex matrices with i.i.d.~complex standard normal entries
\cite{MR0173726}.  When $\phi(z) = \max\{1, |z|\}^{-s}$ for $s =
N + 1$, then $P_N$ defines the joint density of the roots of random
complex polynomials of degree $N$ chosen uniformly from the set of
such polynomials with Mahler measure at most 1 (Mahler measure is a
measure of complexity of polynomials; see
Section~\ref{sec:mahler-measure} below).

The ensembles \eqref{eq:0} can be
put in the context of two-dimensional electrostatics by envisioning
the random variables $z_1, \ldots, z_N$ as a system of repelling
particles confined to the plane, and placed in the presence of an
attractive potential $V: \C \rightarrow [0,\infty)$ which keeps the
particles from fleeing to infinity.  When such a system is placed in
contact with a heat reservoir at a particular temperature, the
location of the particles is random, and the joint density of
particles is given by $P_N$ for $\phi(z) = e^{-V(z)}$.  The connection
between eigenvalues of random matrices and particle systems in 2D
electrostatics is originally attributed to Dyson \cite{Dys62}, and is central in the treatise
\cite{forrester-book}.  From this perspective, it makes sense to
investigate $P_N$ for different naturally-arising confining potentials.

Real ensembles are different in that the roots/eigenvalues of the real polynomials/matrices are either real or come in complex conjugate pairs. Hence, the joint density for such ensembles is not defined on $\C^N$, but rather on
\[
\bigcup_{L + 2M = N} \R^L \times \C^M,
\]
where the union is over all pairs of non-negative integers such that
$L + 2M = N$.  For each such pair, we specify a {\em partial} joint
density $P_{L,M} : \R^L \times \C^M \rightarrow [0,\infty)$ given by
\begin{equation}
\label{eq:3}
P_{L,M}(\mathbf x, \mathbf z) = \frac{1}{Z_N} \bigg\{ \prod_{\ell=1}^L
\phi(x_{\ell}) \bigg\} \bigg\{ \prod_{m=1}^M \phi(z_m)
\phi(\overline{z}_m)\bigg\} |\Delta(\mathbf x \vee \mathbf z \vee
\overline{\mathbf z})|,
\end{equation}
where $\mathbf x \vee \mathbf z \vee \overline{\mathbf z}$ is the
vector formed by joining to $\mathbf x$ all the $z_m$ and their
complex conjugates, and $Z_N$ is the normalization constant given by
\[
Z_N = \sum_{(L,M)} \frac{1}{L! M!} \int_{\R_L} \int_{\C^M}
P_{L,M}(\mathbf x, \mathbf z) \upd \mu_{\R}^L(\mathbf x) \upd
\mu_{\C}^M(\mathbf z).
\]
(Here $\mu_{\R}^L$ and $\mu_{\C}^M$ are Lebesgue measure on $\R^L$
and $\C^M$ respectively).  Note that we may assume that $\phi(z) =
\phi(\overline z)$, since otherwise we could replace $\phi(z)$ with
$\sqrt{\phi(z) \phi(\overline{z})}$, without changing the partial
joint densities.

In this work we focus on the case $\phi(z)=e^{-V(z)}$, where the confining potential $V(z)$ is the scaled logarithmic  equilibrium potential of
the interval $[-2,2]$, see Section~\ref{ssec:pot-theory}. That is,
\begin{equation}
\label{eq1}
V(z) = s \log \frac{|z + \sqrt{z^2 - 4}|}2, \qquad s > N,
\end{equation}
where we take the principal branch of the square root. We view $s > N$
as a parameter of the system and we will call the point process on
$\C$ whose joint density is given by \eqref{eq:0} and \eqref{eq1} the
\emph{complex reciprocal Mahler ensemble}.  Similarly, the joint
densities given by \eqref{eq:3} and \eqref{eq1} define a process on
roots of real polynomials, and we will call this process the
\emph{real reciprocal Mahler process}. The reason we call these point
processes Mahler ensembles is they can be interpreted as choosing
polynomials uniformly at random from starbodies of Mahler measure when
viewed as distance functions (in the sense of the geometry of numbers)
on coefficient vectors of reciprocal polynomials.

\subsection{Mahler Measure}
\label{sec:mahler-measure}
The Mahler measure of a polynomial $f(z) \in \C[z]$ is given by
\[
M\left( a \prod_{n=1}^N (z - \alpha_n) \right) = |a| \prod_{n=1}^N
\max\big\{1, | \alpha_n | \big\}.
\]
This is an example of a {\em height} or measure of complexity of
polynomials, and arises in number theory when restricted to
polynomials with integer (or otherwise algebraic) coefficients.  There
are many examples of heights (for instance norms of coefficient
vectors), but Mahler measure has the attractive feature that it is
multiplicative: $M(f g) = M(f) M(g)$.

There are many open (and hard) questions revolving around the range of
Mahler measure restricted to polynomials with integer coefficients.
Perhaps the most famous is {\em Lehmer's question} which asks whether
1 is a limit point in the set of Mahler measures of integer
polynomials \cite{MR1503118}.  Since cyclotomic polynomials all have
Mahler measure equal to 1, it is clear that 1 is in this set; it is unknown
whether there is a `next smallest' Mahler measure, though to date the
current champion in this regard provided by Lehmer himself is
\[
z^{10} + z^9 - z^7 - z^6 - z^5 - z^4 - z^3 + z + 1,
\]
whose Mahler measure is approximately $1.18$.

A reciprocal polynomial is a polynomial whose coefficient vector is
palindromic; that is, $f(z)$ is a degree $N$ {\em reciprocal} polynomial if
$z^N f(1/z) = f(z)$.  Clearly if $\alpha$ is a root of $f$, then so
too is $1/\alpha$, from which the name `reciprocal' arises.
Reciprocal polynomials arise in the number theoretic investigation of
Lehmer's conjecture since, as was shown by Smyth in the early 1970s,
the Mahler measure of a non-reciprocal integer polynomial is
necessarily greater than approximately 1.3 \cite{MR0289451}.  Thus,
many questions regarding the range of Mahler measure reduce to its
range when restricted to reciprocal integer polynomials.

Another question regarding the range of Mahler measure concerns
estimating the number of integer polynomials (or the number of
reciprocal integer polynomials) of fixed degree $N$ and Mahler measure
bounded by $T > 0$ as $T \rightarrow \infty$.  Such questions were
considered in \cite{MR1868596} and \cite{sinclair-2005}, and the main
term in this estimate depends on the volume of the set of real degree
$N$ polynomials (or real reciprocal polynomials) whose Mahler measure is
at most 1.  This set is akin to a unit ball, though it is not convex.
It is from here that our interest in the zeros of random polynomials
chosen uniformly from these sorts of `unit balls' arose. The
(non-reciprocal) case was studied in \cite{MR2903124} and
\cite{SYreal}, and here we take up the reciprocal case.

In order to be precise we need to specify exactly what we mean by
choosing a reciprocal polynomial {\em uniformly} from those with
Mahler measure at most 1. For $\lambda \in [0,\infty)$ we define the
$\lambda$-homogeneous Mahler measure by
\begin{equation}
\label{mahler}
M_{\lambda}\left( a \prod_{n=1}^N (z - \alpha_n) \right) =
|a|^{\lambda} \prod_{n=1}^N \max\big\{1, | \alpha_n | \big\}.
\end{equation}
To treat polynomials with complex and real coefficients simultaneously, we shall write $\mathbb K$ to mean $\C$ or $\R$ depending on the considered case. Identifying the coefficient vectors of degree $N$ polynomials with
$\mathbb K^{N+1}$, we also view $M_{\lambda}$ as a function on $\mathbb K^{N+1}$,
and define
\[
B_{\lambda}(\mathbb K) = \big\{ \mathbf a \in \mathbb K^{N+1} : M_{\lambda}(\mathbf a)
  \leq 1 \big\}.
\]
These are the degree $N$ unit-{\em starbodies} for Mahler measure.  We can then define the reciprocal unit starbody
as the intersection of the subspace of reciprocal polynomials with
the $B_{\lambda}(\mathbb K)$.  However, as the set of reciprocal polynomials has
Lebesgue measure zero in $\mathbb K^{N+1}$, this is not an
optimal definition for the purposes of selecting a polynomial
uniformly from this set.  In order to overcome this difficulty we need
some natural parametrization of the set of reciprocal polynomials.

If $N$ is odd, and $f$ is reciprocal, then $-f(-1) = f(-1)$.  That is,
$-1$ is always a root of an odd reciprocal polynomial, and
$f(z)/(z+1)$ is an even degree reciprocal polynomial.  Thus, when
considering the roots of random reciprocal polynomials, it suffices to
study even degree reciprocal polynomials.  By declaring that
$M_{\lambda}(z^{-1}) = 1$ and demanding that $M_{\lambda}$ be
multiplicative, we can extend Mahler measure to the algebra of Laurent
polynomials $\C[z, z^{-1}]$, and we define a {\em reciprocal Laurent}
polynomial to be one satisfying $f(z^{-1}) = f(z)$.  Notice that
reciprocal Laurent polynomials form a sub-algebra of Laurent
polynomials.  We can map the set of degree $2N$ reciprocal polynomials
on a set of reciprocal Laurent polynomials via the map
$f(z) \mapsto z^{-N} f(z)$, and the $\lambda$-homogeneous Mahler
measure is invariant under this map.  We will call the image of this
map the set of degree $2N$ reciprocal Laurent polynomials (the leading
monomial is $z^N$, but there are $2N$ zeros).

Now observe that if $f(z) \in \C[z]$ is a degree $N$ polynomial, then $f(z +1/z)$ is a degree $2N$ reciprocal Laurent polynomial and conversely, any reciprocal Laurent polynomial is an algebraic polynomial in $z+1/z$. Hence, we define the {\em reciprocal} Mahler measure of $f(z)$ to be the Mahler measure of $f(z + 1/z)$. Specifically, let $M_{\lambda}^{\rec}: \C[z] \rightarrow [0,\infty)$ be defined by $M_{\lambda}^{\rec}(f(z)) := M_{\lambda}(f(z + 1/z))$. As before, if we identify the set of degree $N$ polynomials with $\mathbb K^{N+1}$ and
define the reciprocal starbodies to be
\[
B^{\rec}_{\lambda}(\mathbb K) = \big\{ \mathbf a \in \mathbb K^{N+1} : M^{\rec}_{\lambda}(\mathbf a)
  \leq 1 \big\}.
\]
The real/complex reciprocal Mahler ensemble is the point process on $\C$ induced by choosing a polynomial uniformly from
$B_{\lambda}^{\rec}(\mathbb K)$ for $\mathbb K=\R$ or $\C$.  It was observed by the first author that the joint density function of such a process is given by \eqref{eq:0} and \eqref{eq1} in the complex case  \cite{MR2145532} and by \eqref{eq:3} and \eqref{eq1} in the real case \cite{sinclair-2005} with $s=(N+1)/\lambda$. Without going into the details we just mention that the factors $|\Delta(\mathbf z)|^2$ and $|\Delta(\mathbf x \vee \mathbf z \vee \overline{\mathbf z})|$ in \eqref{eq:0} and \eqref{eq:3}, respectively, come from the Jacobian of the change of variables from the coefficients of polynomials to their roots and $\phi(z)=e^{-V(z)}$ with $V(z)$ as in \eqref{eq1} appears because
\begin{eqnarray*}
f\left(z + \frac1z\right) &=& a \prod_{n=1}^N \left(z + \frac{1}{z} -  \alpha_j\right) \\
&=& a \prod_{n=1}^N \frac1z \left(z - \frac{\alpha_n+\sqrt{\alpha_n^2-4}}2\right) \left(z - \frac{\alpha_n-\sqrt{\alpha_n^2-4}}2\right),
\end{eqnarray*}
where we take the principal branch of the square root, which necessarily yields that
\begin{equation}
\label{mahler-rec}
M_{\lambda}^{\rec}(f) = |a|^\lambda \prod_{n=1}^N \frac{|\alpha_n+\sqrt{\alpha_n^2-4}|}2.
\end{equation}

\subsection{Mahler Measures and Logarithmic Potentials}
\label{ssec:pot-theory}

Mahler measure and the reciprocal Mahler measure can put into the more
general framework of {\em multiplicative distance functions} formed
with respect to a given compact set $K\subset \C$. Indeed, given a
compact set $K$, it is known that the infimum of the logarithmic
energies
\[
I[\nu] := -\iint\log|z-w|\upd\nu(z)\upd\nu(w)
\]
taken over all probability Borel measures supported on $K$ is either infinite ($K$ cannot support a measure with finite logarithmic energy; such sets are called \emph{polar}) or is finite and is achieved by a unique minimizer, say $\nu_K$, which is called the \emph{equilibrium distribution} on $K$. The logarithmic capacity of $K$ is set to be zero in the former case and $e^{-I[\nu_K]}$ in the latter. It is further known that the function
\[
V_K(z) := I[\nu_K] + \int\log|z-w|\upd\nu_K(w)
\]
is positive and harmonic in $\C\setminus K$ and is zero on $K$ except perhaps on a polar subset. Assume for convenience that $K$ has capacity $1$, i.e., $I[\nu_K]=0$. Then we can define the multiplicative distance function with respect to $K$ on algebraic polynomials by
\[
M(f;K) = \exp\left\{\int\log|f(z)|\upd\nu_K(z)\right\} =
|a|\prod_{n=1}^N e^{-V_K(\alpha_n)},
\]
where \( f(z)=a\prod_{n=1}^N (z-\alpha_n) \). When  $\overline\C\setminus K$ is simply connected, $V_K(z)$ is in fact continuous in $\C$ and is given by  $\log |\Phi_K(z)|$ which is continued by $0$ to $K$, where $\Phi_K$ is a conformal map of the complement of $K$ to the complement of the closed unit disk such that $\Phi_K(\infty)=\infty$. Then
\[
M(f;K) = |a|\prod_{\alpha_n\not\in K}|\Phi_K(\alpha_n)|.
\]
One can easily check now that the Mahler measure $M_1$, see \eqref{mahler}, and the reciprocal Mahler measure $M_1^\rec$, see \eqref{mahler-rec}, are equal to $M\big(\cdot;\overline\D\big)$ and $M\big(\cdot;[-2,2]\big)$, respectively.

\subsection{Determinantal and Pfaffian Point Processes}

Everything in this section is standard, but we include it for
completeness.  The interested reader might consult \cite{MR2129906,
  MR2552864} and \cite{borodin-2008} to get a more detailed
explanation of the complex and real cases, respectively.

Given a Borel set $B \subset \C$ and a random vector $\mathbf z$ chosen according to \eqref{eq:0}, we define the random
variable $N_B$ to be the number of points in $\mathbf z$ that belong to $B$.
The $n$th {\em intensity} or {\em correlation} function of the
ensemble \eqref{eq:0} is a function $R_n : \C^n \rightarrow [0,\infty)$ such that for
disjoint Borel sets $B_1, \ldots, B_n$,
\begin{equation}
\label{expected-complex}
E[N_{B_1} \cdots N_{B_n}] = \int_{B_1 \times \cdots \times B_n}
R_n(\mathbf z) \; \upd\mu_{\C}^n(\mathbf z).
\end{equation}
Correlation functions are the basic objects from which probabilities
of interest are computed.  Ensembles with joint densities of the form
(\ref{eq:0}) are {\em determinantal}, that is, there exists a {\em kernel} $K_N : \C
\times \C \rightarrow [0, \infty)$ such that for all $n$,
\begin{equation}
\label{det-kernel}
R_n(\mathbf z) = \det\left[ K_N(z_j, z_k) \right]_{j,k=1}^n.
\end{equation}

\begin{theorem}
\label{determinantal}
The kernel \( K_N(z,w) \) for the ensemble \eqref{eq:0} is given by
\[
K_N(z,w) = \phi(z)\phi(w)\sum_{n=0}^{N-1} \pi_n(z)\overline{\pi_n(w)}
\]
where $\pi_n$ are orthonormal polynomials with respect to (w.r.t.) the weight $\phi^2$, i.e., $\int\pi_n(z)\overline{\pi_m(z)}\phi(z)^2\upd\mu_\C(z) = \delta_{nm}$.
\end{theorem}

The situation is a bit more complicated for real ensembles with
partial joint densities given by (\ref{eq:3}) since the expected
number of real roots is positive.  In this case, we define the $(\ell,
m)$-intensity or correlation function to be $R_{\ell, m} : \R^{\ell}
\times \C^m \rightarrow [0, \infty)$ such that for disjoint Borel subsets
$A_1, \ldots, A_{\ell}$ of the real line and disjoint Borel subsets
$B_1, \ldots, B_m$ of the open upper half plane,
\begin{equation}
\label{expected-real}
E[N_{A_1} \cdots N_{A_{\ell}} N_{B_1} \cdots N_{B_m}] = \int_{A_1
  \times \cdots \times A_{\ell}} \int_{B_1 \times \cdots \times B_m} R_{\ell,
  m}(\mathbf x, \mathbf z) \upd \mu_{\R}^{\ell}(\mathbf x) \upd
\mu_{\C}^m(\mathbf z).
\end{equation}
Ensembles with partial joint densities given as in \eqref{eq:3} are
{\em Pfaffian}: there exists a $2 \times 2$ matrix kernel $\mathbf K_N: \C
\times \C \rightarrow \C$ such that
\begin{equation}
\label{MatrixKernels}
R_{\ell,m}(\mathbf x, \mathbf z) = \Pf
\begin{bmatrix}
\left[ \mathbf K_N(x_i, x_j) \right]_{i,j=1}^{\ell} & \left[ \mathbf K_N(x_i, z_n) \right]_{i,n=1}^{\ell,m} \medskip \\
-\left[ \mathbf K_N^{\transpose}(z_k, x_j) \right]_{k,j=1}^{m,\ell} &
\left[ \mathbf K_N(z_k, z_n) \right]_{k,n=1}^{m}
\end{bmatrix}.
\end{equation}
The formula for the kernel depends on the {\em species} (real or complex) of
the argument.  This kernel takes the form
\begin{equation}
\label{real-K_N}
\mathbf K_N(z,w) = \begin{bmatrix}
\kappa_N(z,w) & \kappa_N \epsilon(z,w) \medskip \\
\epsilon \kappa_N(z, w) & \epsilon \kappa_N \epsilon(z,w) + \frac12 \sgn(z-w)
\end{bmatrix},
\end{equation}
where $\sgn(\cdot)$ of a non-real number or zero is set to be zero, $\kappa_N$ is an {\em orto-kernel} $\C \times \C \rightarrow \C$, and
\begin{equation}
\label{eq:22}
\epsilon f(z) := \left\{
\begin{array}{ll}
{\displaystyle \frac{1}{2} \int f(t) \sgn(t - z) \, \upd \mu_\R(t)} &
\mbox{if $z \in \R$,} \medskip \\
\mathrm i \sgn \big(\im(z)\big) f(\overline z) & \mbox{if $z \in \C \setminus \R$,}
\end{array}
\right.
\end{equation}
where when written on the left, as in $\epsilon \kappa_N(z,w)$,
$\epsilon$ acts on $\kappa_N$ as a function of $z$, and when written
on the right it acts on $\kappa_N$ as a function of $w$.

\begin{theorem}
\label{pfaffian}
Let $N=2J$ be even. Then the orto-kernel $\kappa_N(z,w)$ for ensemble \eqref{eq:3} is given by
\begin{equation}
\label{orto-kernel}
\kappa_N(z,w) = 2 \phi(z) \phi(w) \sum_{j=0}^{J-1} \big(\pi_{2j}(z)\pi_{2j+1}(w) - \pi_{2j}(w)\pi_{2j+1}(z) \big),
\end{equation}
where polynomials $\pi_n$, $\deg(\pi_n)=n$, are skew-orthonormal, that is, $\la \pi_{2n} | \pi_{2m} \ra = \la \pi_{2n+1} | \pi_{2m+1} \ra = 0$ and $\la \pi_{2n} | \pi_{2m+1} \ra = - \la \pi_{2m+1} | \pi_{2n} \ra = \delta_{nm}$, w.r.t. the skew-symmetric inner product
\begin{equation}
\label{skew-inner}
\la f  \, | \,  g \ra := \int \big[ (f\phi)(z)\epsilon(g\phi)(z) - \epsilon(f\phi)(z) (g\phi)(z) \big] \, \upd(\mu_{\R} + \mu_{\C})(z).
\end{equation}
\end{theorem}

Note that the skew-orthogonal polynomials are not uniquely defined since one may replace $\pi_{2m+1}$  with $\pi_{2m+1}+c\pi_{2m}$ without disturbing skew-orthogonality. Moreover, if polynomials $\pi_{2n}$ and $\pi_{2n+1}$ are skew-orthonormal, then so are $c\pi_{2n}$ and $\pi_{2n+1}/c$. However, neither of these changes alters the expression \eqref{orto-kernel} for the orto-kernel $\kappa_N$.

When $N$ is odd, there is a formula for the orto-kernel similar to
that given in Theorem~\ref{pfaffian}.  We anticipate that the scaling
limits of the odd $N$ case will be the same as those for the even $N$
case (reported below), but due to the extra complexity (with little
additional gain) we concentrate only on the even $N$ case here.

\section{Main Results}

Henceforth, we always assume that $s$ and $N\leq\lfloor s-1 \rfloor$ are such that the limits
\begin{equation}
\label{constants}
c:=\lim_{N\to\infty}(s-N) \in [1,\infty] \quad \text{and} \quad \lambda:=\lim_{N\to\infty}Ns^{-1}\in[0,1]
\end{equation}
exist. Furthermore, we set $\phi(z) := |\Phi(z)|^{-s}$,  $\Phi(z) = (z+\sqrt{z^2-4})/2$. Notice that $\phi(z)$ is well defined in the whole complex plane.

\subsection{Orthogonal and Skew-Orthogonal Polynomials}

Denote by $K_{N,s}(z,w)$ and $\kappa_{N,s}(z,w)$ the kernels introduced in \eqref{det-kernel} and~\eqref{real-K_N}, respectively, for $\phi$ as above. It follows from Theorem~\ref{determinantal} and \cite[Proposition~2]{MR2903124} that
\begin{equation}
\label{KNs}
K_{N,s}(z,w) = \phi(z)\phi(w)\sum_{n=0}^{N-1}\frac{s^2-(n+1)^2}{2\pi s}U_n(z)\overline U_n(w),
\end{equation}
where $U_n(z)$ is the $n$-th monic Chebysh\"ev polynomial of the second kind for $[-2,2]$, i.e.,
\begin{equation}
\label{Cheb2}
U_n(z):=\Phi^n(z)\Phi^\prime(z)\big(1-\Phi^{-2(n+1)}(z)\big) = \frac{\Phi^{n+1}(z)-\Phi^{-n-1}(z)}{\sqrt{z^2-4}}.
\end{equation}

Similarly, we know from Theorem~\ref{pfaffian} that the orto-kernel $\kappa_{N,s}(z,w)$ is expressible via skew-orthogonal polynomials.
\begin{theorem}
\label{skew-ortho}
Polynomials skew-orthonormal w.r.t. skew-inner product \eqref{skew-inner} with $\phi$ as above are given by
\begin{equation}
\label{skew-ortho2}
\left\{
\begin{array}{lll}
\pi_{2n}(z) & = & \displaystyle \frac{4n+3}{16} C_{2n}^{(3/2)}\left(\frac z2\right), \bigskip \\
\pi_{2n+1}(z) & = & \displaystyle \left(1-\frac{(2n+1)^2}{s^2}\right)C_{2n+1}^{(1/2)}\left(\frac z2\right) - \frac1{s^2}C_{2n+1}^{(3/2)}\left(\frac z2\right),
\end{array}
\right.
\end{equation}
where $C_m^{(\alpha)}(x)$ is the classical ultraspherical polynomial of degree $m$, i.e., it is orthogonal to all polynomials of smaller degree w.r.t. the weight $(1-x^2)^{\alpha-1/2}$ on $[-1,1]$ having $\frac{2^m\Gamma(m+\alpha)}{\Gamma(\alpha)\Gamma(m+1)}$ as the leading coefficient.
\end{theorem}

\subsection{Exterior Asymptotics}

We start with the asymptotic behavior of the kernels in $\overline\C\setminus[-2,2]$.

\begin{figure}[h!]
\centering
\includegraphics[scale=.4]{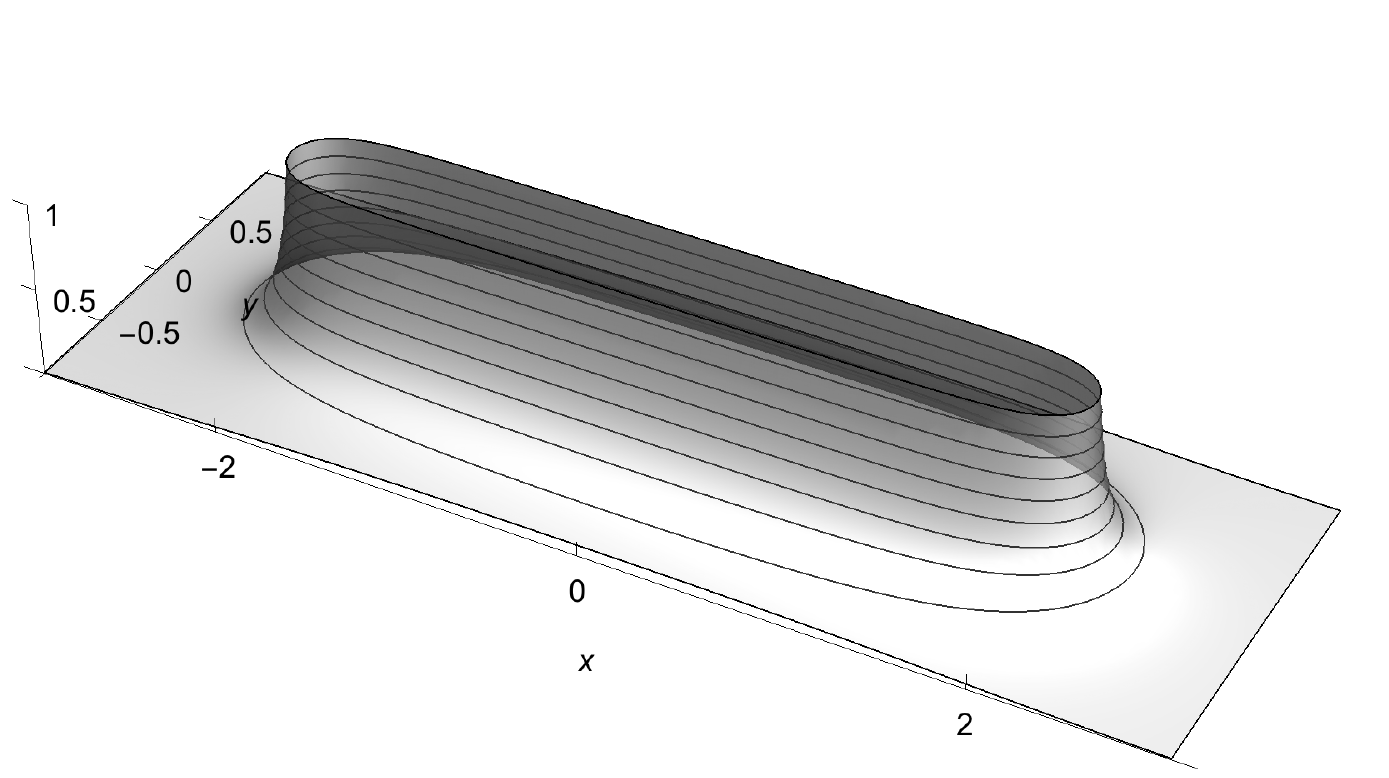}
\includegraphics[scale=.4]{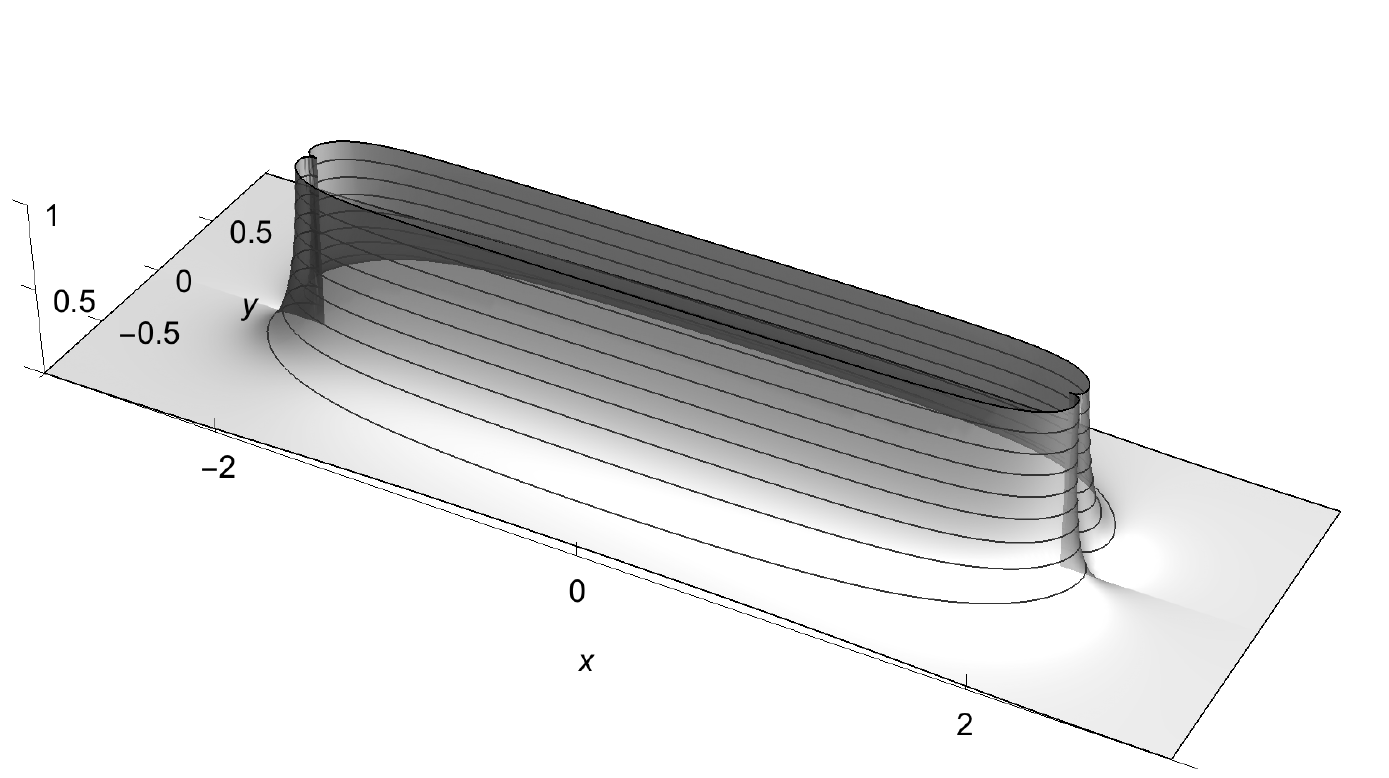}
\begin{caption}{The unscaled limiting spatial density of complex roots near $[-2,2]$ for the complex ensemble on the left and the real ensemble on the right. Notice the cleft along the real axis for the real ensemble. This is due to the fact that roots repel. When the roots come in complex conjugate pairs, this repulsion introduces a paucity of complex zeros near the real axis.}
\label{fig:compdens}
\end{caption}
\end{figure}

\begin{theorem}
\label{complex-exterior}
Assuming \eqref{constants}, it holds that
\begin{equation}
\label{complex-out}
\lim_{N\to\infty} \frac{|\Phi(z)\Phi(w)|^s}{\big(\Phi(z)\overline{\Phi(w)}\big)^N}\frac{K_{N,s}(z,w)}{s-N} = \frac{1+\lambda}{2\pi} \left[1 + \frac{c^{-1}}{\Phi(z)\overline{\Phi(w)}-1}\right] \frac{\Phi^\prime(z)\overline{\Phi^\prime(w)}}{\Phi(z)\overline{\Phi(w)}-1}
\end{equation}
locally uniformly for $z,w\not\in[-2,2]$.
\end{theorem}

It follows from \eqref{complex-out} that the limit of $K_{N,s}(z,w)$  is equal to zero when $c=\infty$, while
\[
\lim_{N\to\infty} K_{N,s}(z,z) = \frac1\pi\frac1{|\Phi(z)|^{2c}} \left[c+\frac1{|\Phi(z)|^2-1}\right] \frac{|\Phi^\prime(z)|^2}{|\Phi(z)|^2-1}
\]
when $c<\infty$. Hence, the expected number of zeros of random polynomials in each open subset of $\overline\C\setminus[-2,2]$ is positive and finite in this case, see \eqref{det-kernel}.

\begin{theorem}
\label{real-exterior}
Let $N$ be even. Assuming \eqref{constants}, it holds that
\begin{multline}
\label{real-out}
\lim_{N\to\infty} \frac{|\Phi(z)\Phi(w)|^s}{\big(\Phi(z)\Phi(w)\big)^N}\frac{\kappa_{N,s}(z,w)}{s-N} = \frac{\lambda(1+\lambda)}{2\pi} \left[1 + \frac{c^{-1}}{\Phi(z)\Phi(w)-1}\right] \frac{\Phi^\prime(z)\Phi^\prime(w)}{\Phi(z)\Phi(w)-1} \times \\
\times \frac{\Phi(w)-\Phi(z)}{\sqrt{\Phi^2(z)-1}\sqrt{\Phi^2(w)-1}}
\end{multline}
locally uniformly for $z,w\not\in[-2,2]$.
\end{theorem}

Theorem~\ref{real-exterior} indicates that $\mathbf K_{N,s}(z,w)$ has a non-zero exterior limit only when $c<\infty$.

\begin{theorem}
\label{matrix-exterior}
Under the conditions of Theorem~\ref{real-exterior}, assume in addition that $c<\infty$. Then it holds that
\begin{equation}
\label{K-exterior}
\lim_{N\to\infty}
\left\{
\begin{array}{rcl}
\epsilon\kappa_{N,s}\epsilon(x,y) &=& F(x,y), \medskip \\
\left(|\Phi(z)|/\Phi(z)\right)^N\kappa_{N,s}\epsilon(z,y) &=& -(\partial_xF)(z,y), \medskip \\
\left(|\Phi(w)|/\Phi(w)\right)^N\epsilon\kappa_{N,s}(x,w) &=& -(\partial_yF)(x,w),
\end{array}
\right.
\end{equation}
locally uniformly for $x,y\in\R\setminus[-2,2]$ and $z,w\not\in[-2,2]$, where
\begin{multline*}
F(x,y) := \frac1\pi\int_{\sgn(x)\infty}^{\Phi(x)}\int_{\sgn(y)\infty}^{\Phi(y)}\left[c + \frac1{uv-1}\right] \frac1{|uv|^c}\frac{u-v}{uv-1} \frac{\upd u\upd v}{\sqrt{u^2-1}\sqrt{v^2-1}}  + \bigskip \\
+  \frac1{\sqrt\pi}\frac{\Gamma(\frac{c+1}2)}{\Gamma(\frac c2)}\left(\sgn(x)\int_{\sgn(y)\infty}^{\Phi(y)}-\sgn(y)\int_{\sgn(x)\infty}^{\Phi(x)}\right)\frac1{|u|^c}\frac{\upd u}{\sqrt{u^2-1}}
\end{multline*}
for $x,y\in\R\setminus[-2,2]$, where $\sqrt{u^2-1}$ is understood to be holomorphic in $\C\setminus[-1,1]$.
\end{theorem}

Even though $F(x,y)$ is defined for real arguments only, its partial derivatives naturally extend to complex ones. Notice also that \( \frac1c(\partial^2_{xy} F)(z,w) \) is equal to the right-hand side of \eqref{real-out}.

\subsection{Scaling Limits in the Real Bulk}

To find the scaling limits of $K_{N,s}(z,w)$ and $\kappa_{N,s}(z,w)$ on $(-2,2)$, it is convenient to compute these limits for $\phi(z)$ separately. Notice that $\phi(y)=1$ whenever $y\in(-2,2)$.

\begin{proposition}
\label{prop:phi1}
Given $x\in(-2,2)$, set $\omega^{-1}(x):=\sqrt{4-x^2}$. Assuming \eqref{constants}, it holds that
\[
\lim_{N\to\infty} \phi\left(x + \frac a{N\omega(x)}\right) = e^{-|\im(a)|/\lambda},
\]
locally uniformly for $x\in(-2,2)$ and $a\in\C$ when $\lambda>0$ and locally uniformly for $x\in(-2,2)$ and $a\in\C\setminus\R$ when $\lambda=0$ (the limit is zero in this case).
\end{proposition}

In light of Proposition~\ref{prop:phi1}, let us write
\begin{equation}
\label{tildeK}
K_{N,s}(z,w) =: \phi(z)\phi(w) \widetilde K_{N,s}(z,w).
\end{equation}

Then the following theorem holds.

\begin{theorem}
\label{complex-bulk}
Assuming \eqref{constants}, it holds that
\begin{equation}
\label{bulk}
\lim_{N\to\infty}\frac1{sN\omega^2(x)}\widetilde K_{N,s}\left(x + \frac a{N\omega(x)},x + \frac b{N\omega(x)} \right) = \frac1\pi\int_0^1\big(1-(\lambda t)^2\big)\cos\big(\big(\overline b-a\big)t\big)\mathrm dt
\end{equation}
locally uniformly for $x\in(-2,2)$ and $a,b\in\C$.
\end{theorem}

In the real case, analogously to \eqref{tildeK}, let us set
\begin{equation}
\label{tildekappa}
\kappa_{N,s}(z,w) =: \phi(z)\phi(w) \tilde \kappa_{N,s}(z,w) \quad \text{and} \quad \kappa_{N,s}\epsilon(z,y)=:\phi(z)\widetilde{\kappa_{N,s}\epsilon}(z,y),
\end{equation}
where \( y\in(-2,2) \) (\( \phi(y)=1 \) in this case). Then the following theorem holds.

\begin{theorem}
\label{real-bulk}
Let $N$ be even. Assuming \eqref{constants}, it holds that
\begin{equation}
\label{bulk1}
\lim_{N\to\infty}\frac1{N^2\omega^2(x)}\tilde \kappa_{N,s}\left(x + \frac a{N\omega(x)},x + \frac b{N\omega(x)} \right) = \frac1\pi\int_0^1 t\big(1-(\lambda t)^2\big)\sin\big((b-a)t\big)\upd t
\end{equation}
locally uniformly for $x\in(-2,2)$ and $a,b\in\C$. Furthermore, it holds that
\begin{equation}
\label{bulk2}
\lim_{N\to\infty}\frac1{N\omega(x)}\widetilde{\kappa_{N,s}\epsilon}\left(x+\frac a{N\omega(x)},x+\frac b{N\omega(x)}\right) =  \frac1\pi\int_0^1\big(1-(\lambda t)^2\big)\cos\big((b-a)t\big)\upd t
\end{equation}
locally uniformly for $x\in(-2,2)$, $a\in\C$, and $b\in\R$. Finally, we have that
\begin{equation}
\label{bulk3}
\lim_{N\to\infty}\epsilon\kappa_{N,s}\epsilon\left(x+\frac a{N\omega(x)},x+\frac b{N\omega(x)}\right) =  \frac1\pi\int_0^1\frac{1-(\lambda t)^2}t\sin\big((b-a)t\big)\upd t
\end{equation}
locally uniformly for $x\in(-2,2)$ and $a,b\in\R$.
\end{theorem}

\begin{figure}[h!]
\centering
\includegraphics[scale=.6]{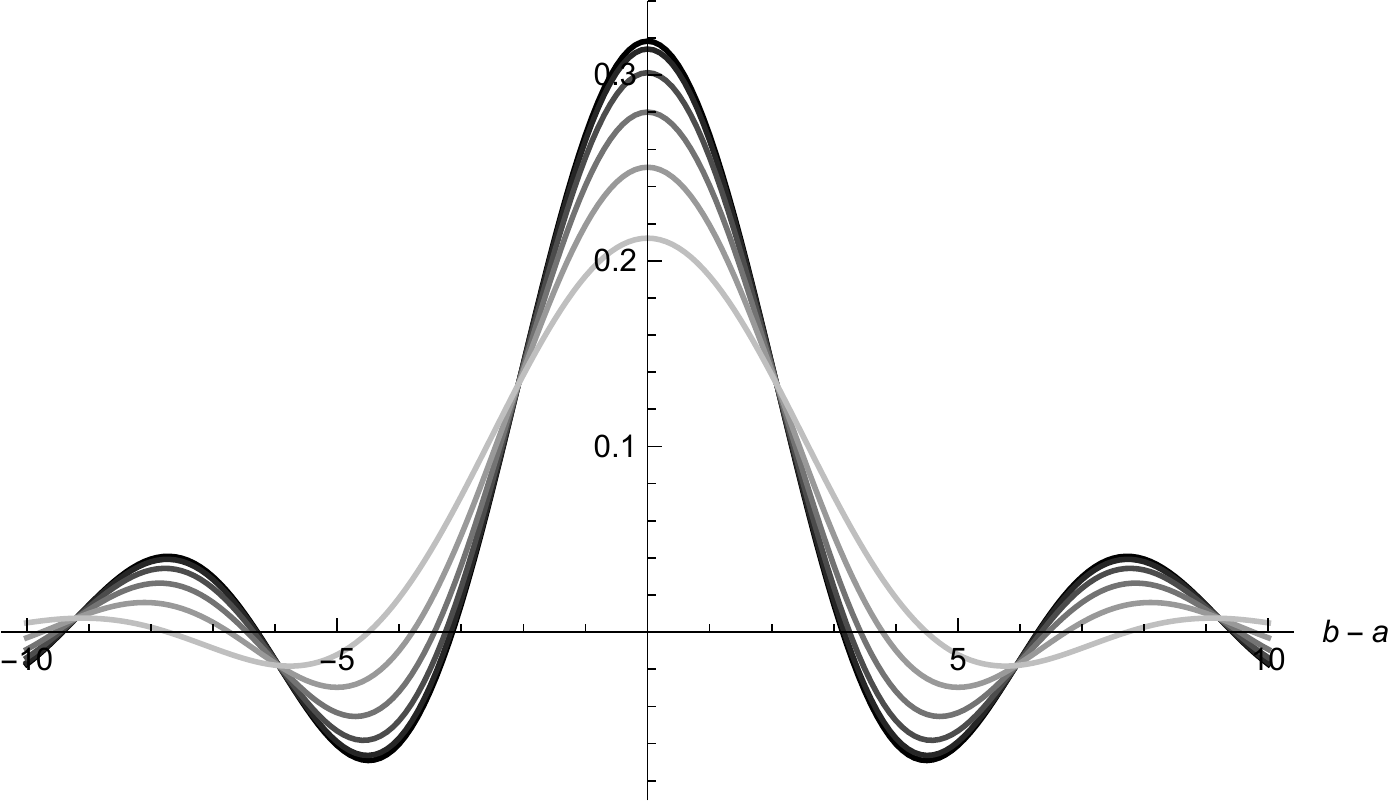}
\begin{caption}{The scaled kernels \eqref{bulk} and \eqref{bulk2} as a function $(b-a) \in \R$ for $\lambda \in [0,1]$. The darkest curve is for $\lambda = 0$ and is equal to the classical sine kernel. Note that for the real ensemble, when $\lambda > 0$ this kernel does not tell us about the local density of real roots, but rather tells us about density of complex roots near the real axis. In this situation, as $\lambda$ increases, the attraction of zeros to $[-2,2]$ decreases and the zeros are more likely to drift into the complex plane. This phenomenon is captured by the decrease in amplitude of the kernel.
}
\label{fig:compker}
\end{caption}
\end{figure}

Notice that knowing limit \eqref{bulk2} is sufficient for our purposes as \( \epsilon\kappa_N(z,w) = -\kappa_N\epsilon(w,z) \) by \eqref{orto-kernel} as the orto-kernel is antisymmetric. Observe also that
\[
\lim_{N\to\infty}\frac1{N\omega(x)}\kappa_{N,s}\epsilon\left(x+\frac a{N\omega(x)},x+\frac b{N\omega(x)}\right) = \frac1\pi \frac{\sin(b-a)}{b-a}
\]
uniformly for \(a,b\in\R\) when \( \lambda =0 \) by \eqref{bulk2} (\( \kappa_{N,s}\epsilon(x,x) \) is exactly the function needed to compute the expected number of real zeros, see \eqref{EN} further below).

\subsection{Scaling Limits at the Real Edge}

Since $\phi(-z)=\phi(z)$, $K_{N,s}(-z,-w)=K_{N,s}(z,w)$, and $\kappa_{N,s}(-z,-w)=-\kappa_{N,s}(z,w)$, we report the scaling limits at $2$ only.

\begin{proposition}
\label{prop:phi2}
Assuming \eqref{constants}, it holds that
\[
\lim_{N\to\infty} \phi\left(2 - \frac {a^2}{N^2}\right) = e^{-|\im(a)|/\lambda}
\]
uniformly on compact subsets of $\C$ when $\lambda>0$, and uniformly on compact subsets of $\C\setminus\R$ when $\lambda=0 $ (the limit is zero).
\end{proposition}

In the case of random polynomials with complex coefficients the following theorem holds.

\begin{theorem}
\label{complex-edge}
Let $\widetilde K_{N,s}(z,w)$ be as  \eqref{tildeK}. Assuming \eqref{constants}, it holds that
\begin{equation}
\label{edge}
\lim_{N\to\infty}\frac1{sN^3}\widetilde K_{N,s}\left(2 - \frac{a^2}{N^2}, 2 - \frac{b^2}{N^2} \right) = \frac1{2\pi a\overline b} \int_0^1\big(1-(\lambda t)^2\big)\sin(at)\sin\big(\overline b t\big)\mathrm dt
\end{equation}
uniformly for $a,b$ on compact subsets of $\C$.
\end{theorem}

Recall that the Bessel functions of the first kind are defined by
\[
J_\nu(z) = \left(\frac z2\right)^\nu \sum_{n=0}^\infty \frac{(-1)^n}{\Gamma(n+1)\Gamma(n+\nu+1)}\left(\frac z2\right)^{2n},
\]
where we take the principal branch of the $\nu$'s power of $z/2$. Since $J_{1/2}(z)=2\sin(z)/\sqrt{2\pi z}$, the right-hand side of \eqref{edge} specializes to
\[
\frac1{2\sqrt{ab}}\frac{J_{1/2}(a)bJ_{1/2}^\prime(b)-J_{1/2}(b)aJ_{1/2}^\prime(a)}{2(a^2-b^2)}
\]
when $\lambda=0$ and $a,b\in\R$, which is a classical Bessel kernel up to the factor $1/2\sqrt{ab}$.

\begin{theorem}
\label{real-edge}
Let $\tilde\kappa_{N,s}(z,w)$ and $\widetilde{\kappa_{N,s}\epsilon}(z,y)$ be as in \eqref{tildekappa} and $N$ be even. Assuming \eqref{constants}, it holds that
\begin{equation}
\label{edge1}
\lim_{N\to\infty}\frac1{N^4}\tilde\kappa_{N,s}\left(2-\frac{a^2}{N^2},2-\frac{b^2}{N^2}\right) = \frac1{8ab}\int_0^1t\big(1-(\lambda t)^2\big)\mathbb J_{1,1}(at,bt)\upd t
\end{equation}
uniformly for $a,b\in\C$, where $\mathbb J_{1,1}(u,v):=J_1(u)vJ_0(v)-J_1(v)uJ_0(u)$. Furthermore, we have that
\begin{equation}
\label{edge2}
\lim_{N\to\infty}\frac1{N^2}\widetilde{\kappa_{N,s}\epsilon}\left(2-\frac{a^2}{N^2},2-\frac{b^2}{N^2}\right) = \frac1{4a}\int_0^1\big(1-(\lambda t)^2\big)\mathbb J_{1,2}(at,bt)\upd t
\end{equation}
uniformly for $a\in\C$ and $b\in\R$, where $\mathbb J_{1,2}(u,v):=J_1(u)vJ_1(v)+J_0(v)uJ_0(u)$. Finally, it holds that
\begin{equation}
\label{edge3}
\lim_{N\to\infty}\epsilon\kappa_{N,s}\epsilon\left(2-\frac{a^2}{N^2},2-\frac{b^2}{N^2}\right) = \frac12\int_0^1\frac{1-(\lambda t)^2}t \mathbb J_{2,2}(at,bt)\upd t
\end{equation}
uniformly for $a,b\in\R$, where $\mathbb J_{2,2}(u,v):=uJ_1(u)J_0(v)-vJ_1(v)J_0(u)$.
\end{theorem}

\begin{figure}[h!]
\centering
\includegraphics[scale=.4]{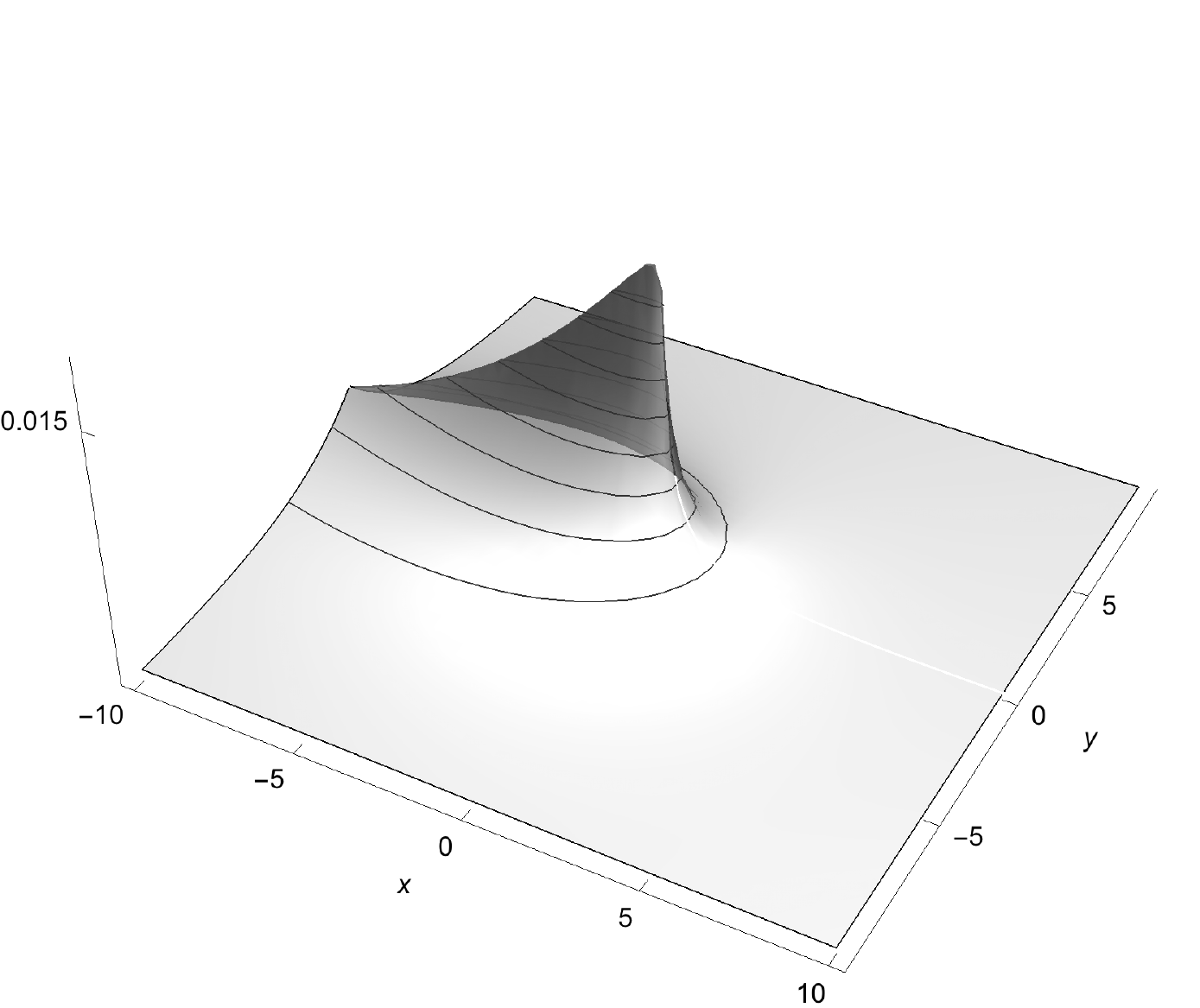}
\includegraphics[scale=.4]{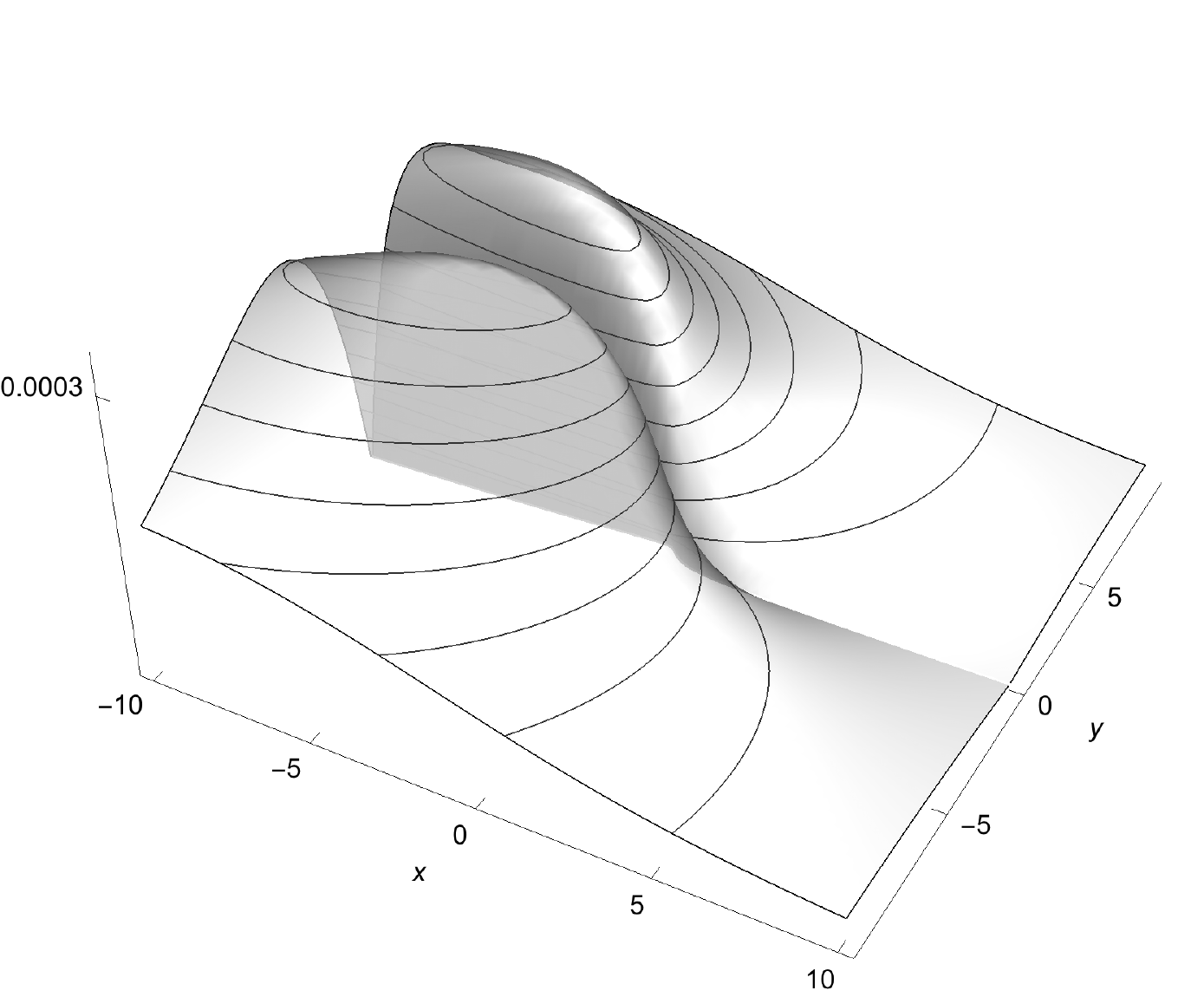}
\begin{caption}{The scaled spatial density of complex roots near $2$ for the complex ensemble on the left and the real ensemble on the right. Here we see a desire for roots to accumulate near $2$ (the origin) with a sharper decrease in the density as we move along the positive $x$-axis (away from the bulk) than along the negative $x$-axis (into the bulk). The difference between the ensembles is starkest in the $y$ direction, and we can see the competition between the attraction to $2$ caused by the potential and the repulsion from the $x$-axis caused by the repulsion between complex conjugate pairs of roots. These images are produced from \eqref{edge} and \eqref{edge2}.}
\label{fig:compdens}
\end{caption}
\end{figure}

Equations \eqref{edge2} and \eqref{edge3} do not cover the cases $b\in\mathrm i \R$ and $a,b\in\mathrm i\R$, respectively.  Such limits exist and we do derive formulas for them, see \eqref{edge2a} and \eqref{edge3a} in Lemma~\ref{2.9-5}. Unfortunately, these formulas are much more cumbersome, which is the reason they are not presented here.

\subsection{Expected Number of Real Zeros}

The zeros of polynomials with real coefficients are either real or come in conjugate symmetric pairs. Hence, one of the interesting questions about such polynomials is the expected number of real zeros. Given a closed set $A\subset\R$, denote by $N_A$ the number of real roots belonging to $A$ of a random degree $N$ polynomial chosen from the real reciprocal Mahler ensemble. Then
\begin{equation}
\label{EN}
E[N_A] = \int_A \Pf\begin{bmatrix}
0 & \kappa_{N,s} \epsilon(x,x) \medskip \\
\epsilon \kappa_{N,s}(x,x) & 0
\end{bmatrix}\upd \mu_\R(x) = \int_A\kappa_{N,s}\epsilon(x,x)\upd\mu_\R(x)
\end{equation}
by \eqref{expected-real}, \eqref{MatrixKernels}, \eqref{real-K_N}, and the anti-symmetry of $\kappa_{N,s}(z,w)$ and $\epsilon\kappa_{N,s}\epsilon(z,w)$. Moreover, the following theorem holds.

\begin{theorem}
\label{expected}
Let $N_\mathsf{in}$ be the number of real roots on $[-2,2]$ of a random degree $N=2J$ polynomial chosen from the real reciprocal Mahler ensemble. Then
\begin{equation}
\label{Ein}
E[N_\mathsf{in}] = N\left[1-\frac{(N+1)(2N+1)}{6s^2}\right].
\end{equation}
Furthermore, let $N_\mathsf{out}$ be the number of real roots on $\R\setminus(-2,2)$ of the said polynomial. Then
\begin{equation}
\label{Eout}
E[N_\mathsf{out}] \sim  - \log\left(1-Ns^{-1}\right),
\end{equation}
where \(f(N,s)\sim g(N,s)\) if there exists \( C>1 \) such that \( C^{-1} \leq f(N,s)/g(s,N)\leq C \).
\end{theorem}

\section{Proofs}

\subsection{Proof of Theorem~\ref{skew-ortho}}

We start by representing polynomials from \eqref{skew-ortho2} as series in Chebysh\"ev polynomials.

\begin{lemma}
\label{lem:poly-sum}
Let $\pi_{2n}$ and $\pi_{2n+1}$ be given by \eqref{skew-ortho2}. Then it holds that
\begin{equation}
\label{skew-ortho1}
\left\{
\begin{array}{lll}
\pi_{2n}(z) & = & \displaystyle \frac{n+\frac34}{2\pi} \sum_{i=0}^n (2i+1)\frac{\Gamma(n-i+\frac12)}{\Gamma(n-i+1)}\frac{\Gamma(n+i+\frac32)}{\Gamma(n+i+2)}U_{2i}(z), \medskip \\
\pi_{2n+1}(z) & = & \displaystyle -\frac1{2\pi}\sum_{i=0}^n (2i+2)\left(1-\frac{(2i+2)^2}{s^2}\right)\frac{\Gamma(n-i-\frac12)}{\Gamma(n-i+1)}\frac{\Gamma(n+i+\frac32)}{\Gamma(n+i+3)}U_{2i+1}(z),
\end{array}
\right.
\end{equation}
where $U_n$ is the degree $n$ monic Chebysh\"ev polynomials of the second kind for $[-2,2]$, see \eqref{Cheb2}.
\end{lemma}
\begin{proof}
Observe that $C_k^{(1)}(\frac z2)=U_k(z)$. Hence, it follows from \cite[Eq. (4.10.27)]{Sz} that
\[
C_{2n+\delta}^{(\alpha)}\left(\frac z2\right) = \frac1{\Gamma(\alpha)\Gamma(\alpha-1)}\sum_{i=0}^n (2i+1+\delta)\frac{\Gamma(n-i-1+\alpha)}{\Gamma(n-i+1)}\frac{\Gamma(n+\delta+i+\alpha)}{\Gamma(n+\delta+i+2)} U_{2i+\delta}(z).
\]
The first relation in \eqref{skew-ortho1} can be obtained now by setting $\alpha=\frac32$ and $\delta=0$ in the above formula. The second one also follows from the above representation combined with the identity
\begin{multline*}
-\frac1{2\pi}\left(1-\frac{(2n+1)^2}{s^2}\right)\frac{\Gamma(n-i-\frac12)}{\Gamma(n-i+1)}\frac{\Gamma(n+i+\frac32)}{\Gamma(n+i+3)} - \frac2{\pi s^2}\frac{\Gamma(n-i+\frac12)}{\Gamma(n-i+1)}\frac{\Gamma(n+i+\frac52)}{\Gamma(n+i+3)} \\
= -\frac1{2\pi}\left(1-\frac{(2i+2)^2}{s^2}\right)\frac{\Gamma(n-i-\frac12)}{\Gamma(n-i+1)}\frac{\Gamma(n+i+\frac32)}{\Gamma(n+i+3)}.
\end{multline*}
\end{proof}

Next, we compute the skew-moments of the Chebysh\"ev polynomials of the second kind.

\begin{lemma}
\label{lem:moments}
Let $\big \la \cdot \, | \, \cdot \big\ra$ be the skew-inner product \eqref{skew-inner}. Then it holds that
\[
\big \la U_{m-1} \, | \, U_{n-1} \big \ra = \left\{
\begin{array}{ll}
0, & m+n\text{ is even}, \medskip \\
\displaystyle\frac nm\frac{16s^2}{(n^2-m^2)(s^2-n^2)}, & m\text{ is odd, }n\text{ is even},  \medskip \\
-\big \la U_{n-1} \, | \, U_{m-1}\big \ra,  & m\text{ is even, }n\text{ is odd}.
\end{array}
\right.
\]
\end{lemma}
\begin{proof}
If $f$ and $g$ satisfy $f(\bar z)=\overline{f(z)}$ and $g(\bar z)=\overline{g(z)}$, then it holds that
\begin{multline*}
\la f \, | \, g \ra = \re\left(4\mathrm i\int_{\C_+}f(z)\overline{g(z)}\phi^2(z)\upd\mu_\C(z)\right) \\ + \iint f(x)\phi(x)g(y)\phi(y)\sgn(y-x)d\mu_\R(x)\upd\mu_\R(y).
\end{multline*}
In what follows, we denote the first summand above by $\la f \, | \, g \ra_{\C_+}$ and the second one by $\la f \, | \, g \ra_\R$.

An elementary change of variables and \eqref{Cheb2} yield
\begin{multline}
\label{compl-inner}
\big \la U_{m-1} \, | \, U_{n-1} \big \ra_{\C_+} =  4\frac{1-(-1)^{m-n}}{m-n}\left( \frac1{2s+m+n} - \frac1{2s-m-n} \right) + \\
+ 4\frac{1-(-1)^{m-n}}{m+n} \left( \frac1{2s-m+n} - \frac1{2s+m-n} \right).
\end{multline}

Expression \eqref{compl-inner} is zero when $m$ and $n$ have the same parity. Notice also that if $f$ and $g$ are either both even or both odd, it holds that $\la f \, | \, g \ra_\R = - \la f \, | \, g \ra_\R=0$. Hence, $\big \la U_{m-1} \, | \, U_{n-1} \big \ra=0$ when $m$ and $n$ have the same parity. Moreover, since $\big \la U_{m-1} \, | \, U_{n-1} \big \ra = - \big \la U_{n-1} \, | \, U_{m-1} \big \ra$, we only need to consider the case where $m$ is odd and $n$ is even. In this situation, we get
\begin{eqnarray*}
\big \la U_{m-1} \, | \, U_{n-1} \big \ra_\R &=& 2\int_0^\infty\int_0^\infty U_{m-1}(x)\phi(x)U_{n-1}(y)\phi(y)\big( \sgn(y-x)+1 \big)\upd\mu_\R(x)\upd\mu_\R(y) \\
& = & 4\int_0^\infty\int_0^y U_{m-1}(x)\phi(x)U_{n-1}(y)\phi(y)\upd\mu_\R(x)\upd\mu_\R(y).
\end{eqnarray*}
Moreover, since $\phi(x)\equiv1$ on $(0,2)$ and $\phi(x)=\Phi^{-s}(x)$ for $x>2$, it holds that
\begin{equation}
\label{intU}
\int_0^y U_{m-1}(x)\phi(x)\upd\mu_\R(x) = \left\{
\begin{array}{ll}
T_m(y)/m, & y\in[0,2], \medskip \\
\displaystyle \frac{2s^2}{m(s^2-m^2)} + \frac{\Phi^{-s-m}(y)}{s+m} - \frac{\Phi^{-s+m}(y)}{s-m}, & y\in(2,\infty),
\end{array}
\right.
\end{equation}
where $T_m(z)=\Phi^m(z)+\Phi^{-m}(z)$, $m\geq1$, is the degree $m$ monic Chebysh\"ev polynomial of the first kind for $[-2,2]$. Therefore, $\big \la U_{m-1} \, | \, U_{n-1} \big \ra_\R$ is equal to the sum of
\begin{equation}
\label{real-part1}
\frac4m\int_0^2 T_m(y)U_{n-1}(y)\upd\mu_\R(y) = \frac{16n}{m(n^2-m^2)}
\end{equation}
and of
\begin{multline}
\label{real-part2}
4\int_2^\infty U_{n-1}(y)\left( \frac{2s^2\Phi^{-s}(y)}{m(s^2-m^2)} + \frac{\Phi^{-2s-m}(y)}{s+m} - \frac{\Phi^{-2s+m}(y)}{s-m} \right)\upd\mu_\R(y) = \\
\frac nm \frac{16s^2}{(s^2-n^2)(s^2-m^2)} + \frac4{s+m}\left(\frac1{2s+m-n} - \frac1{2s+m+n} \right) \\ - \frac4{s-m}\left(\frac1{2s-m-n}-\frac1{2s-m+n}\right).
\end{multline}
By grouping corresponding $2s$ terms in \eqref{compl-inner} and \eqref{real-part2}, simplifying the resulting expression, and then adding \eqref{real-part1}, we get
\[
\frac{16mn}{(n^2-m^2)(s^2-m^2)} + \frac{16n}{m(n^2-m^2)}=\frac nm\frac{16s^2}{(n^2-m^2)(s^2-m^2)}.
\]
The claim now follows by adding the first term on the right-hand side of \eqref{real-part2} to the above expression.
\end{proof}

The next lemma is a technical result that we need to continue with the proof of Theorem~\ref{skew-ortho}.

\begin{lemma}
\label{lem:two-sums}
Put
\begin{equation}
\label{Gammas}
\Gamma_{2n,i} := \frac{\Gamma(n-i+\frac12)}{\Gamma(n-i+1)}\frac{\Gamma(n+i+\frac32)}{\Gamma(n+i+2)} \quad \text{and} \quad \Gamma_{2n+1,i} := \frac{\Gamma(n-i-\frac12)}{\Gamma(n-i+1)}\frac{\Gamma(n+i+\frac32)}{\Gamma(n+i+3)}.
\end{equation}
Given $a\in\C$, it holds that
\begin{equation}
\label{sum1}
\sum_{i=0}^n\frac{4\Gamma_{2n,i}}{(2a)^2-(2i+1)^2} = \frac\pi{2a}\frac{\Gamma(n+a+1)\Gamma(a-n-\frac12)}{\Gamma(n+a+\frac32)\Gamma(a-n)}
\end{equation}
and that
\begin{equation}
\label{sum2}
\sum_{i=0}^n\frac{(2i+2)^2\Gamma_{2n+1,i}}{(2i+2)^2-(2a+1)^2} = \frac{(2a+1)\pi}4\frac{\Gamma(n+a+1)\Gamma(a-n-\frac12)}{\Gamma(n+a+\frac52)\Gamma(a-n+1)}.
\end{equation}
Moreover,
\begin{equation}
\label{sum3}
\sum_{i=0}^n \Gamma_{2n,i} = \frac\pi2 \quad \text{and} \quad \sum_{i=0}^n(2i+2)^2\Gamma_{2n+1,i} = -2\pi.
\end{equation}
\end{lemma}
\begin{proof}
We use the Pochhammer symbol, defined by
\[
x_{(\ell)} := \frac{\Gamma(x + \ell)}{\Gamma(x)} = x(x+1) \cdots (x+\ell-1),
\]
and the elementary transformation, \( (1 - x)_{\ell} = (-1)^{\ell} \G{x}/\G{x - \ell} \) to write the right hand side of (\ref{sum1}) as
\begin{multline*}
\frac\pi{2a}\frac{\Gamma(n+a+1)\Gamma(a-n-\frac12)}{\Gamma(a-n)\Gamma(n+a+\frac32)}
\\ = \frac{\pi}{2} \frac{\G{n+a+1}}{\G{1+a}} \frac{\G{a}}{\G{a - n}}
\frac{\G{a + \frac12}}{\G{n + 1 + a + \frac12}} \frac{\G{a + \frac12 - (n +1)}}{\G{a + \frac12}}.
\end{multline*}
It is obvious that the above expression is a
rational function in $a$ and we can write it as
\begin{equation}
\label{rhs:1}
-\frac{\pi}{2} \frac{(1+a)_{(n)} (1-a)_{(n)}}{\left(\frac12 +
    a\right)_{(n+1)} \left(\frac12 - a\right)_{(n+1)}}.
\end{equation}
The rest of the proof of (\ref{sum1}) is partial fractions
decomposition of this rational function. Its poles are simple and are  located at half
integers with the residue of the pole at $a =m + 1/2$ being given by
\begin{align*}
&\frac{\pi}2 \frac{(1+a)_{(n)} (1-a)_{(n)}}{\left(\frac12 +
    a\right)_{(n+1)}}
\bigg\{\prod_{\ell=0}^{m-1} \frac12 - a + \ell \bigg\}^{-1}
\bigg\{\prod_{\ell=m+1}^{n} \frac12 - a + \ell \bigg\}^{-1} \bigg|_{a
  = m + \frac12} \\
&= \frac{\pi}{2} \frac{(m+\frac32)_{(n)}
  (-m+\frac12)_{(n)}}{(m+1)_{(n+1)}}
\frac{(-1)^{m}}{m!(n-m)!} \\
&= \frac{\pi}{2} \frac{\G{m + n + \frac32}}{\G{m + \frac32}}
\frac{\G{n-m+\frac12}}{\G{-m+\frac12}} \frac{\G{m+1}}{\G{m+n+2}}
\frac{(-1)^m}{m!(n-m)!} \\
&= \frac{\pi}{2m+1} \frac{\G{m + n + \frac32}}{\G{m + \frac12}}
\frac{\G{n-m+\frac12}}{\G{-m+\frac12}} \frac{\G{m+1}}{\G{m+n+2}}
\frac{(-1)^m}{m!(n-m)!} \\
&= \frac{1}{2m+1} \frac{\G{m+n+\frac32} \G{n-m+\frac12}}{\G{m+n+2}
  \G{n-m+1}},
\end{align*}
where we used the formula $\Gamma(z)\Gamma(1-z) = \pi/\sin(\pi z)$ to write
\[
\Gamma\left(m+\frac12\right)\Gamma\left(\frac12-m\right) =
\frac{\pi}{\sin(\pi(m+\frac12))} = (-1)^m \pi.
\]
Since (\ref{rhs:1}) is even in $a$, the residue of the pole at $a =
-m-1/2$ must differ from that of $a =
m+1/2$ by a minus sign.  We have identified all the poles and
their residues, and hence there exists a polynomial $f$ such
that
\begin{align*}
& \frac{\pi}{2} \frac{(1+a)_{(n)} (1-a)_{(n)}}{\left(\frac12 +
    a\right)_{(n+1)} \left(\frac12 - a\right)_{(n+1)}} \\
& = f(a) + \sum_{m=0}^n \frac{1}{2m+1} \frac{\G{m+n+\frac32}
  \G{n-m+\frac12}}{\G{m+n+2} \G{n-m+1}} \left[ \frac{1}{a -
    (m+\frac12)} - \frac{1}{a + (m+\frac12)} \right] \\
& = f(a) + \sum_{m=0}^n \frac{\G{m+n+\frac32}
  \G{n-m+\frac12}}{\G{m+n+2} \G{n-m+1}} \left[ \frac{4}{(2a)^2 -
    (2m+1)^2} \right].
\end{align*}
It remains to show that $f(a)$ is identically 0. But this is trivial,
since, multiplying both sides by $a^2$ and taking the limit as $a
\rightarrow \infty$, we expect the left-hand side to go to a non-zero
constant.  Equality can only be preserved under this procedure when
$f(a) = 0$.

The proof of (\ref{sum2}) is essentially identical, and we only record
the most salient maneuvers.
\begin{align}
\label{eq:4}
\frac{(2a+1)\pi}4\frac{\Gamma(n+a+1)\Gamma(a-n-\frac12)}{\Gamma(a-n+1)\Gamma(n+a+\frac52)}
&= -\frac{\pi}{2}\frac{(1+a)_{(n)} (-a)_{(n)}}{(\frac32 + a)_{(n+1)}
  (\frac12-a)_{(n+1)}}.
\end{align}
This is a rational function with poles at $a = -\frac32, -\frac52,
\ldots, -\frac32 -n$ and $a = \frac12, \frac32, \ldots, \frac12 + n$.
Since this rational function is invariant under the substitution $a
\mapsto -1 - a$, the residues at $-(2m+1)/2$ and $(2m-1)/2$ are equal
up to sign, the former being given by
\[
\frac{m}2 \frac{\G{n+1 -\frac{2m+1}{2}} \G{n +
    \frac{2m+1}{2}}}{\G{n+m+2} \G{n-m+2}}.
\]
The proof of (\ref{sum2}) follows by using these facts to write
(\ref{eq:4}) in partial fractions form and simplifying.

Finally, the first sum in \eqref{sum3} is equal to
\[
\frac\pi2 \lim_{a\to\infty}a\frac{\Gamma(n+a+1)\Gamma(a-n-\frac12)}{\Gamma(n+a+\frac32)\Gamma(a-n)} = \frac\pi2
\]
and the second one can be computed analogously.
\end{proof}

Now, we are ready to prove Theorem~\ref{skew-ortho}.

\begin{lemma}
Under the conditions of Theorem~\ref{skew-ortho}, \eqref{skew-ortho2} holds.
\end{lemma}
\begin{proof}
Let $\pi_n$ be given by \eqref{skew-ortho2}. It follows immediately from Lemmas~\ref{lem:poly-sum} and~\ref{lem:moments} that $\big \la \pi_{2n} \, | \, \pi_{2m} \big \ra = \big \la \pi_{2n+1} \, | \, \pi_{2m+1} \big \ra = 0$. Thus, to show that $\pi_n$'s are indeed skew-orthogonal, it sufficises to prove that
\begin{equation}
\label{goal1}
\big \la \pi_{2n} \, | \, U_{2j+1} \big \ra = \big \la U_{2j} \, | \, \pi_{2n+1} \big \ra =0, \qquad j\in\{0,\ldots,n-1\}.
\end{equation}
Using \eqref{Gammas}, we deduce from Lemmas~\ref{lem:poly-sum} and~\ref{lem:moments} applied with $n=2j+2$ and $m=2i+1$ that
\[
\big \la \pi_{2n} \, | \, U_{2j+1} \big \ra = \frac{4n+3}{2\pi}\frac{(2j+2)s^2}{s^2-(2j+2)^2} \sum_{i=0}^n\frac{4\Gamma_{2n,i}}{(2j+2)^2-(2i+1)^2}.
\]
Hence, the desired claim is a consequence of \eqref{sum1} applied with $a=j+1$ as $\frac1{\Gamma(j+1-n)}=0$. Furthermore, it follows again from Lemmas~\ref{lem:poly-sum} and~\ref{lem:moments} applied with $n=2i+2$ and $m=2j+1$ that
\[
\big \la U_{2j} \, | \, \pi_{2n+1} \big \ra = -\frac8\pi\frac1{2j+1}\sum_{i=0}^n \frac{(2i+2)^2\Gamma_{2n+1,i}}{(2i+2)^2-(2j+1)^2}.
\]
Equation \eqref{sum2} applied with $a=j$ yields that the above skew-inner product is $0$. It remains to verify that $\big \la \pi_{2n} \, | \, \pi_{2n+1} \big \ra =1$. It follows from \eqref{goal1} that
\begin{eqnarray*}
\big \la \pi_{2n} \, | \, \pi_{2n+1} \big \ra & = & \frac1{\sqrt\pi}\left(1-\frac{(2n+2)^2}{s^2}\right)\frac{\Gamma(2n+\frac32)}{\Gamma(2n+2)} \big \la \pi_{2n} \, | \, U_{2n+1} \big \ra \\
& = & \frac{2n+2}{\pi^{3/2}}\frac{\Gamma(2n+\frac52)}{\Gamma(2n+2)}\sum_{i=0}^n\frac{4\Gamma_{2n,i}}{(2n+2)^2-(2i+1)^2}.
\end{eqnarray*}
Upon applying \eqref{sum1} with $a=n+1$, the desired result follows.
\end{proof}

\subsection{Proof of Theorem~\ref{complex-exterior}}

Set $u:=\Phi(z)\overline{\Phi(w)}$. Then $|\Phi(z)\Phi(w)|^sK_{N,s}(z,w)/(\Phi^\prime(z)\overline{\Phi^\prime(w)})$ is equal to
\[
\sum_{n=0}^{N-1}\frac{s^2-(n+1)^2}{2\pi s}u^n\left(1+\mathcal O\big(\rho^{-2(n+1)}\big)\right)
\]
uniformly for $|\Phi(z)|,|\Phi(w)|\geq\rho>1$ by \eqref{KNs} and \eqref{Cheb2}.  It can be readily verified that
\begin{multline*}
\sum_{n=0}^{N-1}\frac{s^2-(n+1)^2}{2\pi s(s-N)}u^{n-N} = \sum_{m=1}^N\frac{s^2-N^2-(2N-1)(1-m)-(2-m)(1-m)}{2\pi s(s-N)}u^{-m} \\
= \frac{1+Ns^{-1}}{2\pi}\sum_{m=1}^Nu^{-m} - \frac{2Ns^{-1}-s^{-1}}{2\pi(s-N)}\left(\sum_{m=1}^Nu^{-m+1}\right)^\prime - \frac1{2\pi s(s-N)}\left(\sum_{m=1}^Nu^{-m+2}\right)^{\prime\prime},
\end{multline*}
which converges to
\[
\frac{1+\lambda}{2\pi}\frac1{u-1} + \frac{\lambda}{\pi}\frac{c^{-1}}{(u-1)^2}
\]
uniformly for $|\Phi(z)|,|\Phi(w)|\geq\rho>1$.  This proves \eqref{complex-out} since $\lambda=1$ when $c^{-1}>0$ and clearly
\[
\lim_{N\to\infty}\sum_{n=0}^{N-1}\frac{s^2-(n+1)^2}{2\pi s(s-N)}u^{n-N}\mathcal O\big(\rho^{-2(n+1)}\big) =0
\]
uniformly for $|\Phi(z)|,|\Phi(w)|\geq\rho$.

\subsection{Proof of Theorem~\ref{real-exterior}}

It is known, see \cite[Theorem~8.21.7]{Sz} or \cite[Theorem~2.11]{AY15}, that
\begin{equation}
\label{ultra}
C_n^{(\alpha)}\left(\frac z2\right) = \frac{(n+1)^{\alpha-1}}{\Gamma(\alpha)} \left(1+\mathcal O\left(\frac 1{n+1}\right)\right) \big(\Phi^\prime(z)\big)^\alpha\Phi^n(z)
\end{equation}
locally uniformly in $\overline\C\setminus[-2,2]$, where $\big(\Phi^\prime(z)\big)^\alpha$ is the principal branch. Set, for brevity,
\[
S_N(z,w) := \frac2{s-N}\frac{\sum_{j=0}^{J-1}\pi_{2j}(z)\pi_{2j+1}(w)}{\Phi^\prime(z)\Phi^\prime(w)\big(\Phi(z)\Phi(w)\big)^N}.
\]
Further, put $u=\Phi(z)\Phi(w)$. Then it follows from \eqref{skew-ortho2} and \eqref{ultra} that
\begin{multline*}
S_N(z,w) = \frac{\Phi(w)}{4\pi}\frac{\sqrt{\Phi^\prime(z)}}{\sqrt{\Phi^\prime(w)}}\sum_{j=0}^{J-1}\left(1+\mathcal O\left(\frac 1{2j+1}\right)\right)\frac{4j+3}s\frac{s^2-(2j+1)^2}{s(s-N)}u^{2j-N} -  \\ -\sqrt{\Phi^\prime(z)}\sqrt{\Phi^\prime(w)}\frac{\Phi(w)}{2\pi}\sum_{j=0}^{J-1}\frac{(4j+3)(2j+1)}{s^2(s-N)}\left(1+\mathcal O\left(\frac 1{2j+1}\right)\right)u^{2j-N}.
\end{multline*}
Upon replacing $j$ with $J-j$, the first sum above can be rewritten as
\[
\sum_{j=1}^J\frac{2N-4j+3}s\frac{s+N-2j+1}s\frac{s-N+2j-1}{s-N}\left(1+\mathcal O\left(\frac 1{N-2j+1}\right)\right)u^{-2j}.
\]
It is easy to see that the coefficient next to $u^{-2j}$ is positive and bounded by $Cj$ for some absolute constant $C$. Hence, these sums form a normal family in $|u|>1$ and to find their limit as $N\to\infty$ is enough to take the limits of the individual coefficients. Clearly, the second sum can be treated similarly. Therefore,
\[
\lim_{N\to\infty} S_N(z,w) = \frac{\lambda(1+\lambda)}{2\pi}\frac{\sqrt{\Phi^\prime(z)}}{\sqrt{\Phi^\prime(w)}}\frac{\Phi(w)}{u^2-1}\left(1+\frac1c\frac{u^2+1}{u^2-1}\right) - \sqrt{\Phi^\prime(z)}\sqrt{\Phi^\prime(w)}\frac{\lambda^2}{\pi c}\frac{\Phi(w)}{u^2-1},
\]
where the limit holds locally uniformly in $|u|>1$. Notice that $\Phi^\prime=\Phi^2/(\Phi^2-1)$. Hence, when $c=\infty$, we get that
\[
\lim_{N\to\infty} \big(S_N(z,w) - S_N(w,z) \big) = \frac{\lambda(1+\lambda)}{2\pi}\frac1{\sqrt{\Phi^2(z)-1}\sqrt{\Phi^2(w)-1}}\frac{\Phi(w)-\Phi(z)}{\Phi(z)\Phi(w)-1}.
\]
On the other hand, when $c<\infty$, it holds that $\lambda=1$ and therefore the above limit is equal to
\[
\frac1\pi\frac1{\sqrt{\Phi^2(z)-1}\sqrt{\Phi^2(w)-1}}\frac{\Phi(w)-\Phi(z)}{\Phi(z)\Phi(w)-1}\left(1+\frac{c^{-1}}{\Phi(z)\Phi(w)-1}\right).
\]
The last two formulas and the definition of $S_N(z,w)$ clearly yield  \eqref{real-out}.

\subsection{Proof of Theorem~\ref{matrix-exterior}}

\begin{lemma}
\label{2.3-1}
It holds that
\[
\left\{
\begin{array}{lll}
\epsilon(\phi\pi_{2n+1})(x) &=& -\int_{\xi\infty}^x\pi_{2n+1}(u)\phi(u)\upd u, \medskip \\
\epsilon(\phi\pi_{2n})(x) &=& -\int_{\xi\infty}^x\pi_{2n}(u)\phi(u)\upd u-\frac\xi2\Gamma_n(s),
\end{array}
\right.
\]
where $\xi=\pm1$ and
\[
\Gamma_n(s) := \int_{-\infty}^\infty\phi(x)\pi_{2n}(x)\upd x = \frac s2\left(n+\frac34\right)\frac{\Gamma(\frac s2+n+1)\Gamma(\frac s2-n-\frac12)}{\Gamma(\frac s2+n+\frac32)\Gamma(\frac s2-n)}.
\]
\end{lemma}
\begin{proof}
The second equality in the definition of $\Gamma_n(s)$ follows from \eqref{skew-ortho1}, \eqref{intU}, and \eqref{sum1} applied with $a=s/2$. Moreover, $\int_{-\infty}^\infty\phi(x)\pi_{2n+1}(x)\upd x=0$ since the integrand is an odd function. The claim then is a simple consequence of \eqref{eq:22}.
\end{proof}

\begin{lemma}
\label{2.3-2}
Let $c$ be given by \eqref{constants}. Assuming $c<\infty$, it holds that
\[
\lim_{N\to\infty}\Phi^{-N}(z)\sum_{j=0}^{J-1} \Gamma_j(s)\pi_{2j+1}(z) = \frac1{\sqrt\pi}\frac{\Gamma(\frac{c+1}2)}{\Gamma(\frac c2)}\frac{\Phi^\prime(z)}{\sqrt{\Phi^2(z)-1}}
\]
locally uniformly for $z\not\in [-2,2]$, where $N=2J$. When $c=\infty$, the limit above is  convergent if it is additionally normalized by $\sqrt{s-N}$.
\end{lemma}
\begin{proof}
By changing the index of summation to $J-1-j$ and using \eqref{skew-ortho2} and \eqref{ultra}, we see that we need to compute the limit of
\begin{multline*}
\frac{\sqrt{\Phi^\prime(z)}}{\sqrt\pi}\sum_{j=0}^{J-1}\left(1+\mathcal O\left(\frac1{J-j}\right)\right)\frac{\sqrt{N-2j}}{s}\frac{\Gamma(\frac s2+J-j)\Gamma(\frac s2-J+j+\frac32)}{\Gamma(\frac s2+J-j-\frac12)\Gamma(\frac s2-J+j+1)}\Phi^{-2j-1}(z) - \\
- \frac{(\Phi^\prime(z))^{3/2}}{2\sqrt\pi}\sum_{j=0}^{J-1}\left(1+\mathcal O\left(\frac1{J-j}\right)\right)\frac{(N-2j)^{3/2}}{s}\frac{\Gamma(\frac s2+J-j)\Gamma(\frac s2-J+j+\frac12)}{\Gamma(\frac s2+J-j+\frac12)\Gamma(\frac s2-J+j+1)}\Phi^{-2j-1}(z).
\end{multline*}
Straightforward estimates show that the above sums form normal families and therefore the limiting function can be obtained by simply evaluating the limits of the coefficients. That is, the above expression converges to
\[
\frac{\sqrt{\Phi^\prime(z)}}{\sqrt\pi}\sum_{j=0}^\infty\frac{\Gamma(\frac{c+3}2+j)}{\Gamma(\frac c2+1+j)}\Phi^{-2j-1}(z) - \frac{(\Phi^\prime(z))^{3/2}}{2\sqrt\pi}\sum_{j=0}^\infty\frac{\Gamma(\frac{c+1}2+j)}{\Gamma(\frac c2+1+j)}\Phi^{-2j-1}(z)
\]
locally uniformly for $z\not\in [-2,2]$. Since $\Phi^\prime=\Phi^2/(\Phi^2-1)=\sum_{j=0}^\infty\Phi^{-2j}$, we get that
\begin{multline*}
\frac{\Phi^\prime(z)}2\sum_{j=0}^\infty\frac{\Gamma(\frac{c+1}2+j)}{\Gamma(\frac c2+1+j)}\Phi^{-2j-1}(z) = \sum_{j=0}^\infty\left(\frac12\sum_{k=0}^j\frac{\Gamma(\frac{c+1}2+k)}{\Gamma(\frac c2+1+k)}\right)\Phi^{-2j-1}(z) = \\
= \sum_{j=0}^\infty\left(\frac{\Gamma(\frac{c+3}2+j)}{\Gamma(\frac c2+1+j)}-\frac{\Gamma(\frac{c+1}2)}{\Gamma(\frac c2)}\right)\Phi^{-2j-1}(z),
\end{multline*}
where the last equality can be shown by induction. Using the identity $\Phi^\prime=\Phi^2/(\Phi^2-1)$ once more, we arrive at the first claim of the lemma. The second claim can be shown analogously.
\end{proof}

\begin{lemma}
\label{2.3-3}
Under the conditions of Theorem~\ref{matrix-exterior}, \eqref{K-exterior} holds.
\end{lemma}
\begin{proof}
For $x,y\in\R$, it holds by \eqref{orto-kernel} and Lemma~\ref{2.3-1} that
\begin{multline*}
\epsilon\kappa_{N,s}\epsilon(x,y) = \int_{\xi_x\infty}^x\int_{\xi_y\infty}^y\kappa_{N,s}(u,v)\upd u\upd v \\ +\left(\xi_x\int_{\xi_y\infty}^y-\xi_y\int_{\xi_x\infty}^x\right)\sum_{j=0}^{J-1}\Gamma_j(s)\pi_{2j+1}(u)\phi(u)\upd u,
\end{multline*}
where $\xi_u=\sgn(u)$, $u\in\R\setminus\{0\}$. The first limit in \eqref{K-exterior} now follows from \eqref{real-out}, Lemma~\ref{2.3-2}, the change of variables $\Phi(u)\mapsto u$ and $\Phi(v)\mapsto v$, and upon observing that $\big(|\Phi(u)|/\Phi(u)\big)^N=1$ for $u\in\R\setminus[-2,2]$ as $\Phi(u)$ is real for such $u$ and $N$ is even. The other two limits can be obtained similarly.
\end{proof}

\subsection{Proof of Proposition~\ref{prop:phi1}}

Given a function \( f \) defined in \( \C\setminus[-2,2] \), we denote its traces on \( (-2,2) \) by
\[
f_\pm(x) := \lim_{\epsilon\to0^+}f(x\pm\mathrm i\epsilon), \quad x\in(-2,2).
\]
Set $x_{N,a}:=x+a/(N\omega(x))$. For $\pm\im(a)\geq 0$, we have that
\begin{eqnarray}
(x_{N,a}^2-4)^{1/2} &=& (x^2-4)_\pm^{1/2} + \frac a{N\omega(x)}\frac{2+a/(N\omega(x))}{(x_{N,a}^2-4)^{1/2}+(x^2-4)_\pm^{1/2}} \nonumber \\
\label{2-0}
&=& (x^2-4)_\pm^{1/2} + \frac{a+\mathcal O\big(N^{-1}\big)}{N\omega(x)}
\end{eqnarray}
locally uniformly for $x\in(-2,2)$ and $a\in\C$. Hence, since $(4-x^2)^{1/2} = \mp\mathrm i(x^2-4)_\pm^{1/2}$, we have that
\begin{equation}
\label{2-1}
\Phi(x_{N,a})  =  \Phi_\pm(x) +\frac {a+\mathcal O\big(N^{-1}\big)}{N\omega(x)} =  \Phi_\pm(x)\left(1 \mp \frac{\mathrm ia+\mathcal O\big(N^{-1}\big)}N\right)
\end{equation}
for \( \pm\im(a)>0 \) locally uniformly for $x\in(-2,2)$ and $a\in\C$. Since $2\Phi^{-1}(z)=z-\sqrt{z^2-4}$ and $(x^2-4)_+^{1/2}+(x^2-4)_-^{1/2}\equiv0$ on $(-2,2)$, it holds that
\begin{equation}
\label{2-2}
\Phi^{-1}(x_{N,a})  = \Phi_\mp(x)\left(1 \pm \frac{\mathrm ia+\mathcal O\big(N^{-1}\big)}N\right)
\end{equation}
for \( \pm\im(a)>0 \) from which the claim of the proposition easily follows.

\subsection{Proof of Theorem~\ref{complex-bulk}}

\begin{lemma}
\label{2.5-1}
Let \( d \) be a non-negative integer, \( \{f_n\} \) be a sequence of positive numbers such that \( \lim_{n\to\infty} f_nn^{-d}=1 \), and $\alpha$ be a real constant. Then
\[
\lim_{N\to\infty} \sum_{n=0}^{N-1}\frac{f_n}{N^{d+1}}\left(1 + \frac\eta N\right)^{n+\alpha} = \int_0^1 t^de^{\eta t}\upd t
\]
uniformly for \( \eta \) on compact subsets of \( \C \), where we take the principal $\alpha$-root. The claim remains valid if \( \eta \) is replaced by \( \eta + \epsilon_N \), where \( \epsilon_N \to 0 \) as \( N\to \infty \).
\end{lemma}
\begin{proof}
Set \( f_n^* := (n+d)\cdots(n+1) \), where \( f_n^*=1 \) if \( d=0 \). Then it holds that
\[
\sum_{n=0}^{N-1}\frac{f_n^*}{N^{d+1}}\left(1 + \frac\eta N\right)^{n+\alpha} = \left(1 + \frac\eta N\right)^\alpha \frac{\upd^d }{\upd \eta^d} \left(\frac{\left(1+\frac\eta N\right)^{N+d}-\left(1+\frac\eta N\right)^d}{\eta}\right).
\]
The desired limit then is equal to
\[
\frac{\upd^d }{\upd \eta^d} \left(\frac{e^{\eta}-1}{\eta}\right) = \frac{\upd^d }{\upd \eta^d} \left(\int_0^1 e^{\eta t}\upd t\right) = \int_0^1 t^de^{\eta t}\upd t
\]
uniformly for \( \eta \) on compact subsets of \( \C \). This proves the lemma with \( f_n=f_n^* \).

Now, write \( f_n=f_n^*(1+\delta_n) \), where, clearly, \( \delta_n\to0 \) as \( n\to\infty\). Further, let \( \{M_N\} \) be a sequence such that \( M_N\to\infty \) and \( M_N/N \to 0\) as \( N\to\infty \). Finally, set \( \delta_N^* := \min\{\delta_n:M_N\leq n\leq N\} \) and \( \delta:=\max_n\delta_n \). Then for \( |\eta|\leq C \), we have that
\begin{multline*}
\left|\sum_{n=0}^{N-1}\frac{f_n^*\delta_n}{N^{d+1}}\left(1 + \frac\eta N\right)^n\right| \leq \delta N^{-1}\sum_{n=0}^{M_N-1}(1+C/N)^n + \delta_N^*N^{-1}\sum_{n=0}^{N-1}(1+C/N)^n \\ \leq \delta C^{-1}\big((1+C/N)^{M_N}-1\big) + \delta_N^* C^{-1}\big((1+C/N)^N-1\big) = o(1)
\end{multline*}
as \( N\to\infty \) since \( (1+C/N)^{M_N} - 1 \to 0 \), \( \delta_N^* \to 0 \), and \( (1+C/N)^N\leq e^C \) in this case. This finishes the proof of the lemma.
\end{proof}

\begin{lemma}
\label{2.5-2}
Let \( d \) and \( \{f_n\} \) be as in Lemma~\ref{2.5-1}. Further, let  \( \{\epsilon_N\} \) be a sequence such that \( |\epsilon_N|\leq C/N \) for all \( N\in\N \). Then
\[
\lim_{N\to\infty} N^{-(d+1)}\sum_{n=0}^{N-1}f_n(\eta+\epsilon_N)^n = 0
\]
uniformly for \( \eta \) on compact subsets of \( \T\setminus\{1\} \) and \( C \) on compact subsets of \( [0,\infty) \).
\end{lemma}
\begin{proof}
Let \( f_n^* \) be as in Lemma~\ref{2.5-1}. Then
\[
N^{-(d+1)}\sum_{n=0}^{N-1}f_n^*(\eta+\epsilon_N)^n = N^{-(d+1)}\frac{\upd^d}{\upd \eta^d}\left(\big((\eta+\epsilon_N)^N-1\big)\frac{(\eta+\epsilon_N)^k-1}{\eta+\epsilon_N-1}\right),
\]
where we used the fact that \( \eta+\epsilon_N-1 \neq 0 \) for all \( N \) large enough. Since \( |\eta|=1 \) and \( |\epsilon_N|\leq C/N \), Leibniz rule for the derivatives of the product and a trivial estimate imply that
\[
\left|\frac{\upd^d }{\upd \eta^d}\left(\big((\eta+\epsilon_N)^N-1\big)\frac{(\eta+\epsilon_N)^k-1}{\eta+\epsilon_N-1}\right)\right| \leq C^*N^d,
\]
where \( C^* \) depends on \( C \) and \( |1-\eta| \). This proves the claim of the lemma with \( f_n=f_n^* \). The proof of the general case repeats word for word the proof of the general case in Lemma~\ref{2.5-1}.
\end{proof}

\begin{lemma}
\label{2.5-3}
Assuming \eqref{constants}, limit \eqref{bulk} holds.
\end{lemma}
\begin{proof}
First we shall consider the case $\im(a)\im(b)\neq 0$. Clearly, it follows from \eqref{2-0} that
\begin{equation}
\label{3-1}
\lim_{N\to\infty}(x_{N,a}^2-4)^{-1/2}\overline{(x_{N,b}^2-4)^{-1/2}} = \pm\omega^2(x), \quad \pm\im(a)\im(b)>0,
\end{equation}
where, as before, $x_{N,a}=x+a/(N\omega(x))$. Hence, we see from \eqref{KNs} and \eqref{Cheb2} that we need to compute the limit of
\begin{multline}
\sum_{n=0}^{N-1}\frac{s^2-(n+1)^2}{2\pi s^2N}\left( \big(\Phi(x_{N,a})\overline{\Phi(x_{N,b})}\big)^{n+1} + \big(\Phi^{-1}(x_{N,a})\overline{\Phi^{-1}(x_{N,b})}\big)^{n+1} \right. \\
\left. - \big(\Phi^{-1}(x_{N,a})\overline{\Phi(x_{N,b})}\big)^{n+1} - \big(\Phi(x_{N,a})\overline{\Phi^{-1}(x_{N,b})}\big)^{n+1} \right).
\label{3-2}
\end{multline}
Set $\tau := \mathrm i(a-\overline b)$. It follows from \eqref{2-1} and \eqref{2-2} that \eqref{3-2} can be rewritten as
\begin{multline}
\pm\sum_{n=0}^{N-1}\frac{s^2-(n+1)^2}{2\pi s^2N}\left(\left(1+\frac{\tau+\mathcal O\big(N^{-1}\big)}N\right)^{n+1} + \left(1-\frac{\tau+\mathcal O\big(N^{-1}\big)}N\right)^{n+1} \right. \\ \left. - \left(\Phi_-^2(x) + \mathcal O\big(N^{-1}\big)\right)^{n+1} -  \left(\Phi_+^2(x) + \mathcal O\big(N^{-1}\big) \right)^{n+1}\right)
\label{3-3}
\end{multline}
for $\pm\im(a)\im(b)>0$, where $\mathcal O\big(N^{-1}\big)$ holds locally uniformly on $(-2,2)\times\C^2$. Thus, Lemma~\ref{2.5-1} applied with \( f_n\equiv 1 \) and \( f_n=(n+1)^2 \) yields that
\begin{multline}
\lim_{N\to\infty}\sum_{n=0}^{N-1}\frac{s^2-(n+1)^2}{2\pi s^2N}\left(\left(1+\frac{\tau+\mathcal O\big(N^{-1}\big)}N\right)^{n+1}+\left(1-\frac{\tau+\mathcal O\big(N^{-1}\big)}N\right)^{n+1}\right)  \\
\label{3-6}
  = \frac1\pi\int_0^1\big(1-(\lambda t)^2\big)\frac{e^{\tau t}+e^{-\tau t}}2\upd t = \frac1\pi\int_0^1\big(1-(\lambda t)^2\big)\cos\big(-\mathrm i\tau t\big)\upd t.
\end{multline}
As $\Phi_\pm^2(x)\neq1$ and $|\Phi_\pm(x)|=1$, Lemma~\ref{2.5-2} applied with \( f_n=1 \) and \( f_n=(n+1)^2 \) lets us conclude that
\[
\lim_{N\to\infty}\sum_{n=0}^{N-1}\frac{s^2-(n+1)^2}{2\pi s^2N}\left(\Phi_\pm^2(x)+\mathcal O\big(N^{-1}\big)\right)^{n+1} = 0.
\]
The desired claim now follows from the last two limits, \eqref{3-1}, and \eqref{3-3}. When either $\im(a)=0$ or $\im(b)=0$, \eqref{complex-bulk} can be deduced similarly.
\end{proof}

\subsection{Proof of Theorem~\ref{real-bulk}}

\begin{lemma}
\label{2.6-1}
Let $j_1,j_2\in\Z$, and $p(\cdot)$ be a monic polynomial of degree $d$. Then
\begin{multline*}
\lim_{N\to\infty}
\frac1{N^{d+1}\omega^2(x)} \sum_{j=0}^{J-1} p(2j) C_{2j+j_1}^{(3/2)}\left(\frac{x_{N,u}}2\right) C_{2j+j_2}^{(1/2)}\left(\frac{x_{N,v}}2\right) = \\ = \mp\frac{\mathrm i}\pi\int_0^1t^d\left(\Phi_\pm^{j_1-j_2+1}(x)e^{\mp\tau t} - \Phi_\mp^{j_1-j_2+1}(x)e^{\pm\tau t}\right)\upd t
\end{multline*}
for $\pm\im(u)\geq0$, where $\tau:=\mathrm i(u-v)$, $x_{N,u}=x+u/(N\omega(x))$, and $N=2J$.
\end{lemma}
\begin{proof}
Using the asymptotic formulas for $C_n^{(\alpha)}(x)$ in a neighborhood of a point $x\in(-1,1)$, see \cite[Theorem~8.21.8]{Sz} or \cite[Theorem~2.11]{AY15}, we get that
\begin{equation}
\label{ultra-bulk}
C_n^{(\alpha)}\left(\frac{x_{N,u}}2\right) =   \frac{(n+1)^{\alpha-1}}{\Gamma(\alpha)} \left(\frac{\Phi^{n+\alpha}(x_{N,u}) + e^{\pm\alpha\pi\mathrm i}\Phi^{-n-\alpha}(x_{N,u})}{(x_{N,u}^2-4)^{\alpha/2}}  + \mathcal O\left(\frac 1{n+1}\right)\right)
\end{equation}
for $\pm\im(u)>0$ and locally uniformly for $x\in(-2,2)$ and $u\in\C$. Observe that
\begin{equation}
\label{4-1}
\lim_{N\to\infty}(x_{N,u}^2-4)^{-3/4}(x_{N,v}^2-4)^{-1/4} = \omega^2(x)\left\{ \begin{array}{rl} -1, & \im(u)\im(v) >0, \smallskip \\ \mp\mathrm i, &\im(u)\im(v) < 0, \end{array} \right.
\end{equation}
when $\pm\im(u)>0$. Thus, to compute the desired limit it is enough to compute the limit of
\begin{multline}
\label{4-2}
\frac1\pi\sum_{j=0}^{J-1}\frac{2p(2j)}{N^{d+1}}\left(\Phi^{2j+j_1+3/2}(x_{N,u}) -\sgn(\im(u))\mathrm i\Phi^{-2j-j_1-3/2}(x_{N,u})\right) \times \\ \times \left(\Phi^{2j+j_2+1/2}(x_{N,v}) +\sgn(\im(v)) \mathrm i\Phi^{-2j-j_2-1/2}(x_{N,v})\right).
\end{multline}

Assume that $\im(u)\im(v)>0$. Then it follows from \eqref{2-1}, \eqref{2-2}, and Lemma~\ref{2.5-2} that the the limit of \eqref{4-2} as \( N\to\infty \) is equal to the one of
\begin{equation}
\label{4-3}
\pm\frac{\mathrm i}\pi\sum_{j=0}^{J-1} \frac{2p(2j)}{N^{d+1}} \left( \frac{\Phi^{2j+j_1+3/2}(x_{N,u})}{\Phi^{2j+j_2+1/2}(x_{N,v})} - \frac{\Phi^{2j+j_2+1/2}(x_{N,v})}{\Phi^{2j+j_1+3/2}(x_{N,u})} \right)
\end{equation}
for $\pm\im(u)>0$. Using \eqref{2-1} and \eqref{2-2} once more we can rewrite \eqref{4-3} as
\begin{equation}
\label{4-4}
\pm\frac{\mathrm i}\pi\sum_{j=0}^{J-1} \frac{\tilde p(j)}{J^{d+1}} \left( \frac{\left(1 \mp \frac{\tau+\mathcal O(N^{-1})}J\right)^{j+\frac{2j_1+3}4}}{\Phi^{j_2-j_1-1}(x_{N,v})}  - \frac{\left(1 \pm \frac{\tau+\mathcal O(N^{-1})}J\right)^{j+\frac{2j_2+1}4}}{\Phi^{j_1-j_2+1}(x_{N,u})}  \right)
\end{equation}
for $\pm\im(u)>0$, where $\tilde p(\cdot)$ is a monic polynomial of degree $d$ and $\mathcal O\big(N^{-1}\big)$ holds locally uniformly in $(-2,2)\times\C^2$. Since $\Phi_+(x)\Phi_-(x)\equiv1$ for $x\in(-2,2)$, the claim of the lemma follows from \eqref{4-1} and Lemma~\ref{2.5-1}.

The case $\im(u)\im(v)<0$ is absolutely analogous. When $\im(u)=0$ (resp. $\im(v)=0$) we can replace $\Phi(x_{N,u})$ (resp. $\Phi(x_{N,v})$) with $\Phi_+(x_{N,u})$ (resp. $\Phi_+(x_{N,u})$) and carry the computations as before (notice that replacing $\pm$ with $\mp$ and $\mp$ with $\pm$ on the right-hand side of the limit in the statement of the lemma does not change the said limit).
\end{proof}

\begin{lemma}
\label{2.6-2}
Under the conditions of Theorem~\ref{real-bulk}, limit \eqref{bulk1} holds.
\end{lemma}
\begin{proof}
It follows from Lemma~\ref{2.6-1} applied with \( p(2j)=2j+3/2 \) and \( p(2j) = (2j+3/2)(2j+1)^2 \) that the limit of
\[
\frac1{N^2\omega^2(x)} \sum_{j=0}^{J-1} \frac{4j+3}8 \left(1-\frac{(2j+1)^2}{s^2}\right) C_{2j}^{(3/2)}\left(\frac{x_{N,a}}2\right) C_{2j+1}^{(1/2)}\left(\frac{x_{N,b}}2\right),
\]
as \( N\to\infty \) is equal to
\[
\frac{\mp\mathrm i}{4\pi}\int_0^1 t\big(1-(\lambda t)^2\big)\big( e^{\mp\tau t} - e^{\pm\tau t}\big)\upd t =  \frac1{2\pi}\int_0^1 t\big(1-(\lambda t)^2\big)\sin\big((b-a)t\big)\upd t.
\]
It further follows from \cite[Eq. (18.9.8)]{DLMF} that
\[
C_{2j+1}^{(3/2)}\left(\frac z2\right) = \left(4j+5 - \frac1{4j+5}\right)\frac{C_{2j+1}^{(1/2)}(z/2)-C_{2j+3}^{(1/2)}(z/2)}{4-z^2}.
\] 
Therefore, Lemma~\ref{2.6-1} now yields that the limit of
\[
\frac1{N^2\omega^2(x)} \sum_{j=0}^{J-1} \frac{4j+3}{8s^2}C_{2j}^{(3/2)}\left(\frac{x_{N,a}}2\right) C_{2j+1}^{(3/2)}\left(\frac{x_{N,b}}2\right)
\]
is equal to zero. The desired claim \eqref{bulk1} now follows from Theorem~\ref{skew-ortho} and the above limits with roles of \( a \) and \( b \) interchanged.
\end{proof}

\begin{lemma}
\label{2.6-3}
It holds that
\[
\epsilon(\phi\pi_{2n})(x) = -\frac{4n+3}8 C_{2n+1}^{(1/2)}\left(\frac x2\right)
\]
and
\begin{multline*}
\epsilon(\phi\pi_{2n+1})(x) = -\frac2{4n+3}\left[\left(1-\frac{(2n+2)^2}{s^2}\right)C_{2n+2}^{(1/2)}\left(\frac x2\right) \right. \\ - \left. \left(1-\frac{(2n+1)^2}{s^2}\right)C_{2n}^{(1/2)}\left(\frac x2\right)\right] + \Delta_n(s),
\end{multline*}
for $x\in(-2,2)$, where
\[
\Delta_n(s)= \frac{(-1)^{n+1}}{\sqrt\pi}\frac{\Gamma(n+\frac12)}{\Gamma(n+2)}\left(1+\frac{(-1)^{n+1}}{2\sqrt\pi}\frac{\Gamma(n+\frac12)}{\Gamma(n+2)}-\frac{(2n+1)(2n+2)}{s^2}\right).
\]
\end{lemma}
\begin{proof}
Observe first that
\[
\left\{
\begin{array}{lll}
\epsilon(\phi\pi_{2n+1})(x) &=&-\int_0^x\pi_{2n+1}(u)\phi(u)\upd u + \int_0^\infty \pi_{2n+1}(u)\phi(u)\upd u, \medskip \\
\epsilon(\phi\pi_{2n})(x) &=& -\int_0^x\pi_{2n}(u)\phi(u)\upd u.
\end{array}
\right.
\]
Indeed, by the very definition \eqref{eq:22}, we have that 
\[
(\epsilon f)(x) = -\int_0^x f(t)\upd t -\frac12\int_{-\infty}^0 f(t)\upd t +\frac12\int_0^\infty f(t)\upd t. 
\]
Hence, for even functions $f$ it holds that $(\epsilon f)(x)=-\int_0^x f(t)\upd t$ and for odd functions $f$ it holds that
\[
(\epsilon f)(x)=-\int_0^x f(t)\upd t+\int_0^\infty f(t)\upd t,
\]
from which the claim easily follows. Using \cite[Eq. (18.9.19)]{DLMF} and \cite[Table 18.6.1]{DLMF}, one can also get that
\[
\left\{
\begin{array}{lll}
\displaystyle\int_0^xC_{2n}^{(3/2)}\left(\frac u2\right)\upd u &=&\displaystyle 2C_{2n+1}^{(1/2)}\left(\frac x2\right) \medskip \\
\displaystyle\int_0^xC_{2n+1}^{(3/2)}\left(\frac u2\right)\upd u &=&\displaystyle 2\left[C_{2n+2}^{(1/2)}\left(\frac x2\right)-\frac{(-1)^{n+1}}{\sqrt\pi}\frac{\Gamma(n+\frac32)}{\Gamma(n+2)}\right].
\end{array}
\right.
\]
Since $\phi(u)=1$ for $u\in(-2,2)$, the expression for $\epsilon(\phi\pi_{2n})(x)$ now follows from Theorem~\ref{skew-ortho}. It further follows from \cite[Eq. (18.9.7)]{DLMF} that
\begin{equation}
\label{2n+1}
\pi_{2n+1}(u)=  \frac1{4n+3}\left[\left(1-\frac{(2n+2)^2}{s^2}\right)C_{2n+1}^{(3/2)}\left(\frac u2\right) -\left(1-\frac{(2n+1)^2}{s^2}\right)C_{2n-1}^{(3/2)}\left(\frac u2\right)\right].
\end{equation}
Hence, we only need to compute the constant term. Clearly,
\begin{multline*}
\left(1-\frac{(2n+2)^2}{s^2}\right)\frac{(-1)^{n+1}}{\sqrt\pi}\frac{\Gamma(n+\frac32)}{\Gamma(n+2)} - \left(1-\frac{(2n+1)^2}{s^2}\right)\frac{(-1)^n}{\sqrt\pi}\frac{\Gamma(n+\frac12)}{\Gamma(n+1)} = \\  \frac{4n+3}2\frac{(-1)^{n+1}}{\sqrt\pi}\frac{\Gamma(n+\frac12)}{\Gamma(n+2)}\left(1-\frac{(2n+1)(2n+2)}{s^2}\right).
\end{multline*}
Thus, it remains to compute $\int_0^\infty \pi_{2n+1}(u)\phi(u)\upd u$. It follows from \eqref{skew-ortho1} and \eqref{intU} that
\[
\int_0^\infty \pi_{2n+1}(u)\phi(u)\upd u = -\frac1\pi\sum_{i=0}^n\frac{\Gamma(n-i-\frac12)}{\Gamma(n-i+1)}\frac{\Gamma(n+i+\frac32)}{\Gamma(n+i+3)} =  \frac1{2\pi}\left(\frac{\Gamma(n+\frac12)}{\Gamma(n+2)}\right)^2,
\]
where the second equality follows from \eqref{sum2} after taking the limit as $a\to-1/2$ and observing that
\[
\frac{\Gamma(a-n-\frac12)}{\Gamma(a-n+1)} = \frac{\Gamma(n-a)}{\Gamma(n-a+\frac32)}\frac{\sin(\pi(n-a-1))}{\sin(\pi(n+a+\frac12))}. 
\]
\end{proof}

\begin{lemma}
\label{2.6-4}
Assuming \eqref{constants}, it holds that
\[
\lim_{N\to\infty} N^{-1}\sum_{n=0}^{J-1}\Delta_n(s)\pi_{2n}(x_{N,a}) = 0
\]
locally uniformly for $x\in(-2,2)$ and $a\in\C$, where $x_{N,a}=x+a/(N\omega(x))$.
\end{lemma}
\begin{proof}
It follows from Theorem~\ref{skew-ortho} and \eqref{ultra-bulk} that we are computing the limit of
\[
\sum_{n=0}^{J-1}\frac{|\Delta_n(s)|(4n+3)\sqrt{2n+1}}{8\sqrt\pi N}\left(\frac{\Phi^{2n+3/2}(x_{N,a}) \mp \mathrm i \Phi^{-2n-3/2}(x_{N,a})}{(-1)^{n+1}(x_{N,a}^2-4)^{3/4}}  + \mathcal O\left(\frac 1{n+1}\right)\right).
\]
The claim of the lemma follows from the asymptotic behavior of the Gamma function, \eqref{constants}, \eqref{2-1}, \eqref{2-2}, Lemma~\ref{2.5-2}, and a simple estimate for the sum involving \( \mathcal O\big((n+1)^{-1}\big) \).
\end{proof}

\begin{lemma}
\label{2.6-5}
Given \eqref{constants}, limits \eqref{bulk2} and \eqref{bulk3} hold.
\end{lemma}
\begin{proof}
Set, as usual, $x_{N,a}=x+a/(N\omega(x))$. Put $K_{n,i}:=1-\frac{(2n+i)^2}{s^2}$. Then \( \widetilde{\kappa_{N,s}\epsilon}(x_{N,a},x_{N,b}) \) is equal to
\begin{eqnarray*}
\frac14\sum_{j=0}^{J-1} C_{2j+1}^{(1/2)}\left(\frac{x_{N,b}}2\right)\left[K_{j,2}C_{2j+1}^{(3/2)}\left(\frac{x_{N,a}}2\right) - K_{j,1}C_{2j-1}^{(3/2)}\left(\frac{x_{N,a}}2\right) \right] \\ -\frac14\sum_{j=0}^{J-1} C_{2j}^{(3/2)}\left(\frac{x_{N,a}}2\right)\left[K_{j,2}C_{2j+2}^{(1/2)}\left(\frac{x_{N,b}}2\right) - K_{j,1}C_{2j}^{(1/2)}\left(\frac{x_{N,b}}2\right) \right] \\ + 2\sum_{j=0}^{J-1}\Delta_j(s)\pi_{2j}(x_{N,a})
\end{eqnarray*}
by Theorem~\ref{skew-ortho}, \eqref{eq:22}, Lemma~\ref{2.6-3}, and \eqref{2n+1}. Denote the three sums above by $S_{N,1}$, $S_{N,2}$, and $S_{N,3}$. Then it follows from Lemma~\ref{2.6-1} that
\begin{eqnarray*}
\lim_{N\to\infty}\frac{S_{N,1}}{N\omega(x)} &=& \mp\frac{\mathrm i\omega(x)}{4\pi}\int_0^1\big(1-(\lambda t)^2\big)\big(\Phi_\pm(x)-\Phi_\mp(x)\big)\big(e^{\mp\mathrm i(b-a)t} + e^{\pm\mathrm i(b-a)t}\big)\upd t \\
&=& \frac1{2\pi}\int_0^1\big(1-(\lambda t)^2\big)\cos\big((b-a)t\big)\upd t,
\end{eqnarray*}
where one needs to observe that $\Phi_\pm(x)\Phi_\mp(x)\equiv1$ and $\Phi_\pm(x)-\Phi_\mp(x)=\pm\mathrm i/\omega(x)$. Similarly, one can check that $S_{N,2}$ has the same limit and therefore \eqref{bulk2} follows from Lemma~\ref{2.6-4}.

Since $\epsilon\kappa_{N,s}\epsilon(x,y)$ is anti-symmetric (in particular, zero on the diagonal) and $(\epsilon f)^\prime(x)=-f(x)$ for $x$ real, we have that $\epsilon\kappa_{N,s}\epsilon(x,y) = \int_x^y \kappa_{N,s}\epsilon(u,y)\upd u$ for $x,y\in(-2,2)$. Hence, it holds that
\[
\epsilon\kappa_{N,s}\epsilon(x_{N,a},x_{N,b}) = \int_a^b \frac1{N\omega(x)}\kappa_{N,s}\epsilon(x_{N,u},x_{N,b})\upd u
\]
and therefore \eqref{bulk3} easily follows from \eqref{bulk2}.
\end{proof}

\subsection{Proof of Proposition~\ref{prop:phi2}}

Let $y_{N,a}:=2-(a/N)^2$. Since we take the principal branch of the square root, it holds that
\begin{equation}
\label{eq:1}
\sqrt{y_{N,a}^2-4} = \frac2N\sqrt{-a^2(1-(a/2N)^2)} = \mp\frac{2\mathrm ia}N\left(1+\mathcal O\big(N^{-2}\big)\right)
\end{equation}
for \( \pm\im(a)\geq0 \) locally uniformly in $\C$. Hence,
\begin{equation}
\label{eq:2}
\Phi(y_{N,a}) =  1 \mp \frac {\mathrm ia}N + \mathcal O\big(N^{-2}\big) \quad \text{and} \quad \Phi^{-1}(y_{N,a}) =  1 \pm \frac {\mathrm ia}N + \mathcal O\big(N^{-2}\big)
\end{equation}
for $\pm\im(a)\geq0$ uniformly on compact subsets of $\C$, from which the desired claim easily follows.

\subsection{Proof of Theorem~\ref{complex-edge}}

Assume first that $\im(a)\im(b)\neq0$. It follows from \eqref{eq:1} that
\begin{equation}
\label{5-0}
\lim_{N\to\infty} N^{-2}(y_{N,a}^2-4)^{-1/2}\overline{(y_{N,b}^2-4)^{-1/2}} = \frac{\pm1}{4a\overline b}
\end{equation}
for \( \pm\im(a)\im(b)>0 \), where $y_{N,a}=2-(a/N)^2$. Hence, we need to compute the limit of \eqref{3-2} with $x_{N,a}$ and $x_{N,b}$ replaced by $y_{N,a}$ and $y_{N,b}$. To this end, we get from \eqref{eq:2} that
\begin{multline*}
\big(\Phi(y_{N,a})\overline{\Phi(y_{N,b})}\big)^{n+1} + \big(\Phi^{-1}(y_{N,a})\overline{\Phi^{-1}(y_{N,b})}\big)^{n+1} \\ = \left(1+\frac{\tau_1+o(1)}N\right)^{n+1} + \left(1-\frac{\tau_1+o(1)}N\right)^{n+1},
\end{multline*}
and
\begin{multline*}
\big(\Phi^{-1}(y_{N,a})\overline{\Phi(y_{N,b})}\big)^{n+1} + \big(\Phi(y_{N,a})\overline{\Phi^{-1}(y_{N,b})}\big)^{n+1} \\ =  \left(1+\frac{\tau_2+o(1)}N\right)^{n+1} + \left(1-\frac{\tau_2+o(1)}N\right)^{n+1},
\end{multline*}
$\tau_1:=\mathrm i\big(a \mp \overline b\big)$, $\tau_2=\mathrm i\big(a \pm \overline b\big)$ when $\pm\im(a)\im(b)>0$, where $o(1)$ holds locally uniformly on $\C^2$. Hence, we get by \eqref{3-6} that the desired limit is equal to
\begin{multline}
\label{5-1}
\frac1\pi\int_0^1\big(1-(\lambda t)^2\big)\left[ \cos\big(-\mathrm i\tau_1 t\big) - \cos\big(-\mathrm i\tau_2 t\big)\right] \mathrm dt \\ =
\pm \frac2\pi \int_0^1\big(1-(\lambda t)^2\big)\sin(at)\sin\big(\overline b t\big) \mathrm dt,
\end{multline}
for $\pm\im(a)\im(b)>0$. Combining \eqref{5-0} and \eqref{5-1}, we get \eqref{edge}. When either $\im(a)=0$ or $\im(b)=0$, \eqref{edge} can be deduced similarly.

\subsection{Proof of Theorem~\ref{real-edge}}

To prove Theorem~\ref{real-edge}, it will be convenient to set $\tilde J_\nu(z) := (2/z)^\nu J_\nu(z)$, which is always an entire function. It follows from \cite[Eq. (10.9.4)]{DLMF} that
\begin{equation}
\label{5-3}
\tilde J_\nu(z)= \frac1{\sqrt\pi\Gamma(\nu+\frac12)}\int_{-1}^1e^{\mathrm izt}(1-t^2)^{\nu-1/2} \upd t.
\end{equation}

\begin{lemma}
\label{2.9-1}
Let $2(\alpha_1+\alpha_2)\in\N$, $j_1,j_2\in\N$ be fixed, and $p(\cdot)$ be a monic polynomial of degree $d$. Then
\begin{multline}
\label{5-4}
\lim_{N\to\infty}\frac1{N^{d+2\alpha_1+2\alpha_2+1}}K_N^{\alpha_1,\alpha_2,d}\left(1-\frac{a^2}{2N^2},1-\frac{b^2}{2N^2}\right) = \\ = \frac12\frac{\Gamma(\alpha_1+1)}{\Gamma(2\alpha_1+1)}\frac{\Gamma(\alpha_2+1)}{\Gamma(2\alpha_2+1)} \int_0^1 t^{d+2(\alpha_1+\alpha_2)}\tilde J_{\alpha_1}(at)\tilde J_{\alpha_2}(bt)\upd t
\end{multline}
uniformly for $a,b$ on compact subsets of $\C$, where $N=2J$ and
\[
K_N^{\alpha_1,\alpha_2,d}(z,w) := \sum_{j=0}^{J-1}p(2j) C_{2j+j_1}^{(\alpha_1+\frac12)}(z)C_{2j+j_2}^{(\alpha_2+\frac12)}(w).
\]
\end{lemma}
\begin{proof}
It follows from \cite[Eq. (18.10.4) and Table 18.6.1]{DLMF} that
\begin{equation}
\label{5-5}
C_n^{(\alpha+\frac12)}(z) = \frac1{\sqrt\pi}\frac{\Gamma(\alpha+1)}{\Gamma(\alpha+\frac12)\Gamma(2\alpha+1)}\frac{\Gamma(n+2\alpha+1)}{\Gamma(n+1)}\int_{-1}^1\left(z+v\sqrt{z^2-1}\right)^n\upd \mu_\alpha(v),
\end{equation}
for any determination of the square root, where $\upd\mu_\alpha(v):=(1-v^2)^{\alpha-\frac12}\upd v$. If $z=1-a^2/(2N^2)$ and $w=1-b^2/(2N^2)$, then
\[
\left(z+v\sqrt{z^2-1}\right)^2\left(w+u\sqrt{w^2-1}\right)^2 = 1 + \mathrm i\frac{av+bu+\mathcal O\big(N^{-1}\big)}J,
\]
where $\mathcal O\big(N^{-1}\big)$ is uniform with respect to $a,b$ on compact subsets of $\C$ and $t,u\in[-1,1]$, and we do not keep track of the determination of the square roots as it is not important in \eqref{5-5}, (observe also that the substitutions $a\mapsto-a$ and $b\mapsto-b$ do not change either side of \eqref{5-4}). Thus, the limit of the left-hand side of \eqref{5-4} is the same as the limit of
\begin{equation}
\label{5-6}
\frac{C}{\pi\Gamma(\alpha_1+\frac12)\Gamma(\alpha_2+\frac12)}\int_{-1}^1\int_{-1}^1\sum_{j=0}^{J-1}\frac{f_j\left(1+\mathrm i\frac{av+bu+\mathcal O(N^{-1})}J\right)^j}{J^{d+2(\alpha_1+\alpha_2)+1}}\upd\mu_{\alpha_1}(v)\upd\mu_{\alpha_2}(u),
\end{equation}
where  $C:=\frac12\frac{\Gamma(\alpha_1+1)}{\Gamma(2\alpha_1+1)}\frac{\Gamma(\alpha_2+1)}{\Gamma(2\alpha_2+1)}$ and
\[
f_j := \frac{p(2j)}{2^{d+2(\alpha_1+\alpha_2)+1}}\frac{\Gamma(2j+j_1+2\alpha_1+1)}{\Gamma(2j+j_1+1)}\frac{\Gamma(2j+j_2+2\alpha_2+1)}{\Gamma(2j+j_2+1)}.
\] 
Since \( \lim_{j\to\infty} f_j j^{-(d+2(\alpha_1+\alpha_2)+1)}=1 \), it follows from Lemma~\ref{2.5-1} that the limit of \eqref{5-6} is equal to
\[
\frac{C}{\pi\Gamma(\alpha_1+\frac12)\Gamma(\alpha_2+\frac12)}\int_{-1}^1\int_{-1}^1 \int_0^1 t^{d+2(\alpha_1+\alpha_2)} e^{\mathrm i(av+bu)t}\upd t\upd\mu_{\alpha_1}(v)\upd\mu_{\alpha_2}(u).
\]
The claim of the lemma now follows from \eqref{5-3}.
\end{proof}

\begin{lemma}
\label{2.9-2}
Under the conditions of Theorem~\ref{real-edge}, \eqref{edge1} holds.
\end{lemma}
\begin{proof}
It follows from Theorems~\ref{pfaffian} and~\ref{skew-ortho} that we need to evaluate the limit of
\begin{multline*}
\frac1{4N^4}\sum_{j=0}^{J-1}\left(1-\frac{(2j+1)^2}{s^2}\right)\left(2j+\frac32\right)C_{2j}^{(3/2)}\left(1-\frac{a^2}{2N^2}\right)C_{2j+1}^{(1/2)}\left(1-\frac{b^2}{2N^2}\right) \\ - \frac1{4s^2N^4}\sum_{j=0}^{J-1}\left(2j+\frac32\right)C_{2j}^{(3/2)}\left(1-\frac{a^2}{2N^2}\right)C_{2j+1}^{(3/2)}\left(1-\frac{b^2}{2N^2}\right),
\end{multline*}
which is equal to
\[
\frac1{16}\int_0^1 t^3\big(1-(\lambda t)^2\big)\tilde J_1(at)\tilde J_0(bt)\upd t - \frac{\lambda^2}{32}\int_0^1t^5\tilde J_1(at)\tilde J_1(bt)\upd t
\]
by \eqref{5-4}. By swapping the the roles of $a$ and $b$, we get that the limit of the left-hand side of \eqref{edge1} is equal to
\[
\frac1{16}\int_0^1 t^3\big(1-(\lambda t)^2\big)\big(\tilde J_1(at)\tilde J_0(bt)-(\tilde J_0(at)\tilde J_1(bt)\big)\upd t.
\]
The desired result now follows from the identities $J_1(z)=(z/2)\tilde J_1(z)$ and $J_0(z)=\tilde J_0(z)$.
\end{proof}

\begin{lemma}
\label{2.9-3}
It holds that
\[
\left\{
\begin{array}{lll}
\epsilon(\phi\pi_{2n+1})(x) &=& -\int_2^x\pi_{2n+1}(u)\phi(u)\upd u +\frac2{s^2}, \medskip \\
\epsilon(\phi\pi_{2n})(x) &=& -\int_2^x\pi_{2n}(u)\phi(u)\upd u - \frac{4n+3}8.
\end{array}
\right.
\]
\end{lemma}
\begin{proof}
The first relation in the proof of Lemma~\ref{2.6-3} gives us
\[
\left\{
\begin{array}{lll}
\epsilon(\phi\pi_{2n+1})(x) &=& -\int_2^x\pi_{2n+1}(u)\phi(u)\upd u +\int_2^\infty \pi_{2n+1}(u)\phi(u)\upd u, \medskip \\
\epsilon(\phi\pi_{2n})(x) &=& -\int_2^x\pi_{2n}(u)\phi(u)\upd u - \int_0^2\pi_{2n}(u)\phi(u)\upd u.
\end{array}
\right.
\]
The claim of the lemma now follows from \eqref{skew-ortho1}, \eqref{intU}, and \eqref{sum3}.
\end{proof}

\begin{lemma}
\label{2.9-4}
Under conditions of Theorem~\ref{real-edge}, it holds locally uniformly for $u\in\C$ that
\[
\lim_{N\to\infty}\frac1{N^2}\sum_{j=0}^{J-1}\left(\frac{4n+3}{4}\pi_{2j+1}+\frac4{s^2}\pi_{2j}\right)\left(2-\frac{u^2}{N^2}\right) = \frac14\int_0^1t\big(1-(\lambda t)^2\big)J_0(ut)\upd t.
\]
\end{lemma}
\begin{proof}
Repeating the proof of Lemma~\ref{2.9-1}, we can show the following. Let $2\alpha\in\N$, $m\in\N$ be fixed, and $p(\cdot)$ be a monic polynomial of degree $d$. Then
\begin{equation}
\label{single}
\lim_{N\to\infty}\frac1{N^{d+2\alpha+1}} \sum_{j=0}^{J-1}p(2j) C_{2j+m}^{(\alpha+\frac12)}\left(1-\frac{u^2}{2N^2}\right) = \frac12\frac{\Gamma(\alpha+1)}{\Gamma(2\alpha+1)} \int_0^1 t^{d+2\alpha}\tilde J_\alpha(ut)\upd t
\end{equation}
locally uniformly for $u\in\C$. Now, it follows from Theorem~\ref{skew-ortho} that $\frac{4n+3}{4}\pi_{2n+1}(z)+\frac4{s^2}\pi_{2n}(z)$ is equal to
\[
\frac{2j+3/2}{2s^2}\left[C_{2j}^{(3/2)}\left(\frac z2\right) +\big( s^2 - (2n+1)^2 \big) C_{2j+1}^{(1/2)}\left(\frac z2\right) - C_{2j+1}^{(3/2)}\left(\frac z2\right) \right].
\]
Clearly, the claim of the lemma is a consequence of \eqref{single} as $\tilde J_0(z)=J_0(z)$.
\end{proof}

\begin{lemma}
\label{2.9-5}
Under the conditions of Theorem~\ref{real-edge}, limits \eqref{edge2} and \eqref{edge3} hold.
\end{lemma}
\begin{proof}
It follows from \eqref{eq:22} and Lemma~\ref{2.9-3} that
\[
\kappa_{N,s}\epsilon(z,y) = -\int_2^y\kappa_{N,s}(z,u)\upd u+ \phi(z)\sum_{j=0}^{J-1}\left(\frac{4j+3}{4}\pi_{2j+1}+\frac4{s^2}\pi_{2j}\right)(z).
\]
for $y\in\R$. Hence, Proposition~\ref{prop:phi2}, \eqref{edge1}, Lemma~\ref{2.9-4}, and the change of variable $u\mapsto 2-\frac{u^2}{N^2}$ imply that
\begin{multline}
\label{edge2a}
\lim_{N\to\infty}\frac1{N^2}\widetilde{\kappa_{N,s}\epsilon}(y_{N,a},y_{N,b}) = \\ \frac14\int_0^1t\big(1-(\lambda t)^2\big)\left[\frac1a\int_0^be^{-|\im(u)|/\lambda}\mathbb J_{1,1}(at,ut)\upd u+J_0(at)\right]\upd t
\end{multline}
uniformly for $a\in\C$ and $b^2\in\R$, where $y_{N,a}=2-a^2/N^2$. When $b\in\R$, the expression in square parenthesis becomes
\[
\frac1aJ_1(at)\int_0^b utJ_0(ut)\upd u - tJ_0(at)\int_0^bJ_1(ut)\upd u+ J_0(at) = \frac baJ_1(at)J_1(bt) + J_0(at)J_0(bt)
\]
as $J_0^\prime(z)=-J_1(z)$ and $(zJ_1(z))^\prime = zJ_0(z)$, from which \eqref{edge2} is immediate. Similarly, we get from Lemma~\ref{2.9-3} that \( \epsilon\kappa_{N,s}\epsilon(x,y) \) is equal to
\[
 \int_2^x\int_2^y \kappa_{N,s}(u,v)\upd v\upd u + \left(\int_0^y-\int_0^x\right)\phi(u) \sum_{n=0}^{J-1}\left(\frac{4n+3}{4}\pi_{2n+1}+\frac4{s^2}\pi_{2n}\right)(u) \upd u.
\]
The same change of variables allows us to show that the limit of $\epsilon\kappa_{N,s}\epsilon(y_{N,a},y_{N,b})$ is equal to
\begin{multline}
\label{edge3a}
\frac12\int_0^1 t\big(1-(\lambda t)^2\big)\left[\int_0^a\int_0^b e^{-\frac{|\im(u)|+|\im(v)|}\lambda}\mathbb J_{1,1}(ut,vt) \upd v\upd u \right. \\ \left. -\left(\int_0^b-\int_0^a\right) e^{-\frac{|\im(u)|}\lambda}uJ_0(ut)\upd u\right] \upd t.
\end{multline}
When $a,b$ are real and therefore the integrals are evaluated along the real axis, we get that
\[
\int_0^a\int_0^b \mathbb J_{1,1}(ut,vt) \upd v\upd u + \left(\int_0^b-\int_0^a\right) uJ_0(ut)\upd u = \frac{aJ_1(at)J_0(bt)-bJ_1(bt)J_0(at)}t,
\]
which proves \eqref{edge3}.
\end{proof}

\subsection{Proof of Theorem~\ref{expected}}

\begin{lemma}
\label{2.10-1}
Equation \eqref{Ein} holds.
\end{lemma}
\begin{proof}
It follows from \eqref{EN} and Theorem~\ref{pfaffian} that
\[
E[N_\mathsf{in}] = \sum_{n=0}^{J-1}\int_{-2}^2 2\big[(\pi_{2n}\phi)(x)\epsilon(\pi_{2n+1}\phi)(x) - (\pi_{2n+1}\phi)(x)\epsilon(\pi_{2n}\phi)(x)\big]\upd x.
\]
Since $(\epsilon f)^\prime(x)=-f(x)$ for $x$ real, we can rewrite the above equality as
\[
E[N_\mathsf{in}] = -2 \sum_{n=0}^{J-1}\epsilon(\pi_{2n}\phi)(x)\epsilon(\pi_{2n+1}\phi)(x)\big|_{-2}^2 -4\sum_{n=0}^{J-1}\int_{-2}^2(\pi_{2n+1}\phi)(x)\epsilon(\pi_{2n}\phi)(x)\upd x.
\]
Since $\epsilon(\pi_{2n}\phi)(-2)=-\epsilon(\pi_{2n}\phi)(2)$ and $\epsilon(\pi_{2n+1}\phi)(-2)=\epsilon(\pi_{2n+1}\phi)(2)$ by Lemma~\ref{2.6-3}, it follows from Lemma~\ref{2.9-3} that the first sum above is equal to $J(2J+1)/s^2$. Set
\[
I_n := -4\int_{-2}^2(\pi_{2n+1}\phi)(x)\epsilon(\pi_{2n}\phi)(x)\upd x = 8\int_0^2\pi_{2n+1}(x)\int_0^x\pi_{2n}(u)\upd u\upd x,
\]
where the second equality holds by Lemma~\ref{2.6-3} and since $\phi(x)\equiv1$ on $[-2,2]$ (notice that the integrand in the integral that defines \( I_n \) is an even function by Lemma~\ref{2.6-3}). Thus, it follows from Lemma~\ref{lem:poly-sum} and \eqref{intU} that
\[
I_n = \frac{4n+3}\pi \sum_{i=0}^n \frac{\Gamma(n-i+\frac12)}{\Gamma(n-i+1)}\frac{\Gamma(n+i+\frac32)}{\Gamma(n+i+2)}\int_0^2\pi_{2n+1}(x)T_{2i+1}(x)\upd x.
\]
Using notation \eqref{Gammas} and Lemma~\ref{lem:poly-sum} once more, this time with \eqref{real-part1}, we obtain that
\[
I_n = -\frac{4n+3}{2\pi^2} \sum_{j=0}^n\left(\sum_{i=0}^n \frac{4\Gamma_{2n,i}}{(2j+2)^2-(2i+1)^2}\right) (2j+2)^2\left(1-\frac{(2j+2)^2}{s^2}\right)\Gamma_{2n+1,j}.
\]
Further, we deduce from \eqref{sum1} applied with $a=j+1$ that the inner sum is zero for all $j<n$ and for $j=n$ it is equal to $[\pi^{3/2}\Gamma(2n+2)]/[(2n+2)\Gamma(2n+5/2)]$. Hence, $I_n=2(1-(2n+2)^2/s^2)$ and therefore
\[
E[N_\mathsf{in}] = \frac{N(N+1)}{2s^2} + \sum_{n=0}^{J-1} I_n =  \frac{N(N+1)}{2s^2} + N - \frac{N(N+1)(N+2)}{3s^2}
\]
from which the claim of the lemma follows.
\end{proof}

\begin{lemma}
\label{2.10-2}
Equation \eqref{Eout} holds.
\end{lemma}
\begin{proof}
As in Lemma~\ref{2.10-1}, we have that
\[
E[N_\mathsf{out}] = \sum_{n=0}^{J-1}\int_2^\infty 4\big[(\pi_{2n}\phi)(x)\epsilon(\pi_{2n+1}\phi)(x) - (\pi_{2n+1}\phi)(x)\epsilon(\pi_{2n}\phi)(x)\big]\upd x,
\]
where we used Lemma~\ref{2.6-3} to deduce that the integrand is an even function. Since \( \pi_{2n}\phi \) is an even function, it follows from the definition of the \( \epsilon \)-operator in \eqref{eq:22} that
\[
\epsilon(\pi_{2n}\phi)(x) = -\int_0^x (\pi_{2n}\phi)(u)\upd u,  \quad x\geq 0.
\]
Therefore, integration by parts, the above identity, Lemma~\ref{2.9-3}, and Lemma~\ref{2.3-1} yield that
\[
E[N_\mathsf{out}] = -\frac{N(N+1)}{2s^2} + \sum_{n=0}^{J-1}8\int_2^\infty\int_0^x (\pi_{2n+1}\phi)(x)(\pi_{2n}\phi)(u)\upd u\upd x.
\]
We get from \eqref{skew-ortho1}, \eqref{Cheb2}, \eqref{sum2}, and \eqref{sum3} that \( \displaystyle \frac1{s+m}\int_2^\infty\frac{\pi_{2n+1}(x)}{\Phi^{2s+m}(x)}\upd x \) is equal to
\begin{eqnarray*}
 & & -\frac1\pi\sum_{i=0}^n\frac{(2i+2)^2}{s^2(s+m)}\frac{s^2-(2i+2)^2}{(2s+m)^2-(2i+2)^2}\Gamma_{2n+1,i}  \\
&=& \frac2{s^2(s+m)} + \frac{s^2-(2s+m)^2}{\pi s^2(s+m)}\sum_{i=0}^n\frac{(2i+2)^2\Gamma_{2n+1,i}}{(2i+2)^2-(2s+m)^2} \\
& = &\frac2{s^2(s+m)} - \frac{(3s+m)(2s+m)}{4s^2}\frac{\Gamma(s+n+\frac{m+1}2)}{\Gamma(s+n+\frac{m+4}2)}\frac{\Gamma(s-n+\frac{m-2}2)}{\Gamma(s-n+\frac{m+1}2)}.
\end{eqnarray*}
Recall that \( \int_2^\infty\pi_{2n+1}\phi =2/s^2 \), see Lemma~\ref{2.9-3}. Denoting by \( \Gamma^{s,n}_{\frac{m+1}2} \) the last summand  in the equation above, we deduce from \eqref{intU} that
\[
8\int_2^\infty(\pi_{2n+1}\phi)(x)\int_0^x (U_{2i}\phi)(u)\upd u\upd x = \frac{32}{(2i+1)s^2} + 8\big(\Gamma^{s,n}_{-i}-\Gamma^{s,n}_{i+1}\big).
\]
Then it follows from Lemma~\ref{lem:poly-sum} and \eqref{sum3} that
\[
E[N_\mathsf{out}] = \frac{N(N+1)}{2s^2} + \frac1\pi \sum_{n=0}^{J-1}(4n+3)\sum_{i=0}^n(2i+1)\Gamma_{2n,i}\big(\Gamma^{s,n}_{-i}-\Gamma^{s,n}_{i+1}\big).
\]
By the Gautschi's inequality for the ratio of Gamma functions \cite[Eq.~(5.6.4)]{DLMF} and straightforward estimates using the fact that $n\leq J-1<s/2$, it holds that $\Gamma^{s,n}_{i+1} \lesssim 1/s^3$, where $\lesssim$ means $\leq$ with an absolute constant. Therefore, \eqref{sum3} and simple upper bounds yield that
\begin{equation}
\label{thelastone}
E[N_\mathsf{out}] = \mathcal O(1) + \frac1\pi \sum_{n=0}^{J-1}(4n+3)\sum_{i=0}^{n-1}(2i+1)\Gamma_{2n,i}\Gamma^{s,n}_{-i},
\end{equation}
where similar estimates allowed us to change the index of summation from \( n \) to \( n-1 \) in the inner sum. Using Gautschi's inequality and straightforward estimates once more, one can verify that \( \Gamma^{s,n}_{-i} \sim (s(s-n-i))^{-3/2} \) and that \( \Gamma_{2n,i} \sim (n(n-i))^{-1/2}\) when \( 0\leq i\leq n-1 \). Observe further that
\begin{multline*}
\sum_{i=1}^n \frac1{\sqrt{i}(s-2n+i)^{3/2}} \sim \int_0^n\frac{\upd x}{\sqrt x(s-2n+x)^{3/2}} \\ = \left.\frac{2}{s-2n}\sqrt{\frac x {s-2n+x}}\right|_0^n \sim \frac1{s-2n}\sqrt{\frac ns}.
\end{multline*}
Thus, the inner sum in \eqref{thelastone} then can be estimated as
\[
\frac1{s^{3/2}}\frac1{\sqrt n}\sum_{i=1}^n \frac{2(n-i)+1}{\sqrt{i}(s-2n+i)^{3/2}} \sim \frac n{s^2}\frac1{s-2n}.
\]
After plugging the above expression into \eqref{thelastone}, \eqref{Eout} easily follows.
\end{proof}

\section*{Acknowledgments}

Part of this research project was completed at the American Institute of Mathematics. The research of the first author was supported in part by a grant from the Simons Foundation, CGM-315685. The research of the second author was supported in part by a grant from the Simons Foundation, CGM-354538.

\end{document}